\documentclass[11pt]{article}

\renewcommand{\bf}{\textbf}

\usepackage[utf8]{inputenc}
\usepackage[T1]{fontenc}
\usepackage{lmodern}
\usepackage[a4paper]{geometry}

\usepackage[usenames,dvipsnames]{xcolor}

\usepackage{epigraph}

\usepackage{amsmath, amssymb, amsthm}
\usepackage{stmaryrd}
\usepackage[all]{xy}

\newcommand{\version}{}

\newtheorem{defi}{Definition}[section]
\newtheorem{thm}[defi]{Theorem}

\newtheorem{propo}[defi]{Proposition}
\newtheorem{lem}[defi]{Lemma}
\newtheorem{lemme}[defi]{Lemma}

\newtheorem{rem}[defi]{Remark}

\newtheorem{cor}[defi]{Corollary}

\newcommand{\norme}[1]{\left \lVert #1\right \rVert}
\newcommand{\abs}[1]{\left \lvert #1 \right \rvert}

\newcommand{\set}[1]{\left \{ #1 \right \} }
\newcommand{\ensent}[1]{\left \llbracket #1 \right \rrbracket}

\newcommand{\bbn}{\ensuremath \mathbb{N}}
\newcommand{\bbz}{\ensuremath \mathbb{Z}}
\newcommand{\bbd}{\ensuremath \mathbb{D}}

\newcommand{\bbc}{\ensuremath \mathbb{C}}

\newcommand{\bbp}{\ensuremath \mathbb{P}}

\newcommand{\ddc}{2 i \partial \bar{\partial}}
\newcommand{\dd}{ ~\text{d}}

\renewcommand{\bar}{\overline}

\DeclareMathOperator{\Leb}{Leb }

\DeclareMathOperator{\Lip}{Lip }

\DeclareMathOperator{\Sup}{Supp }

\makeatletter

\@addtoreset{equation}{section}
\makeatother

\title{Directional dimensions of ergodic currents on $\mathbb C \mathbb P (2)$}
\author{Christophe Dupont and Axel Rogue}

\makeindex

\begin{document}

\pagestyle{myheadings}
\maketitle

\begin{abstract}
Let $f$ be a holomorphic endomorphism of $\mathbb P^ 2$ of degree $d \geq 2$. We estimate the local directional dimensions of closed positive currents $S$ with respect to ergodic dilating measures $\nu$.  We infer several applications. The first one shows that the currents $S$ containing a measure of entropy $h_\nu > \log d$ have a directional dimension $>2$, which answers a question by de Th\'elin-Vigny. The second application asserts that the Dujardin's semi-extremal endomorphisms are close to suspensions of one-dimensional Latt\`es maps. Finally, we obtain an upper bound for the dimension of the equilibrium measure, towards the formula conjectured by  Binder-DeMarco. \end{abstract}




\section{Introduction}

This article concerns the ergodic properties of holomorphic endomorphisms of $\bbp^2$, see \cite{dinsib10}. Let $f$ be an endomorphism of  $\bbp^2$ of degree $d \geq 2$. The Green current $T$ is defined as $T:= \lim_n {1 \over d^n} {f^n}^* \omega$, where $\omega$ is the Fubini-Study form of $\bbp^2$. The equilibrium measure is defined as $\mu:= T \wedge T$, this is an ergodic measure of entropy $\log d^2$, its Lyapunov exponents are $\geq {1 \over 2} \log d$, see \cite{BriDuv99, dinsib10}.

We say that an ergodic measure $\nu$ is \emph{dilating} if its Lyapunov exponents are positive. The ergodic measures of entropy $h_\nu > \log d$ are dilating: their exponents are larger than or equal to  ${1\over 2} (h_\nu - \log d)$, see   \cite{det08, dup12}. It is known that the support of every ergodic measure $\nu$ of entropy  $h_\nu > \log d$ is contained in the support of $\mu$, see  \cite{det05, din07}. The article \cite{dup12} constructs such measures by using coding techniques.
 
\subsection{Directional dimensions}\label{DD}

Let $f$ be an endomorphism of  $\bbp^2$ of degree $d \geq 2$. The algebraic subsets of $\bbp^2$ do not contain any ergodic measure  $\nu$ of entropy $h_\nu > \log d$: this comes from the Gromov's iterated graph argument and from the relative variational principle, see \cite{briduv01} and \cite[Section 1.7]{dinsib10}. In this Section, we quantify that property thanks to the Lyapunov exponents of the measures $\nu$. We shall work in the more general setting of $(1,1)$ closed positive currents $S$. Those currents are as big as the algebraic subsets, since we have:
 $$\forall x \in \bbp^2 \ , \ \underline{d_S}(x) := \liminf_{r \to 0} \frac{\log S \wedge \omega (B_x(r))}{\log r} \geq 2 , $$
which comes from $S \wedge \omega (B_x(r)) \leq c(x) r^2$, see \cite[Chapitre 3]{agbook}. We shall use the notation 
$\bar{d_S}(x)$ for the $\limsup$. A drawback of the trace measure $S \wedge \omega$ is that it does not distinguish any specific direction. 
If $Z$ is a holomorphic coordinate in the neighbourhood of $x \in \bbp^2$, one defines the lower local directional dimension of $S$ with respect to $Z$ by
\[   \underline{d_{S,Z}}(x):=  \liminf_{r \to 0} \frac{\log \left[S \wedge (\frac{i}{2} d Z \wedge d \bar{Z})(B_{x}(r))\right]}{\log r}  ,\]
we shall denote the $\limsup$ by $\bar{d_{S,Z}}(x)$. 
Geometrically, the positive measure $S \wedge (\frac{i}{2} d Z \wedge d \bar{Z} )$ is the average with respect to Lebesgue measure of the slices of the current $S$ transversaly to the direction $Z$, see Proposition~\ref{tranchesdecourant}. If $(Z,W)$ are holomorphic coordinates near $x$, the directional dimensions of $S$ are related to the dimension of $S$ by
\begin{equation}\label{minco} 
\underline{d_S}(x) = \min \set{ \underline{d_{S,Z}}(x),\underline{d_{S,W}}(x) } \ , \ 
\bar{d_S}(x) = \min \set{ \bar{d_{S,Z}}(x),\bar{d_{S,W}}(x) } ,  
\end{equation}
see Proposition \ref{minmesures}. 
 We start with showing lower and upper bounds for the directional dimensions of the Green current $T$ (Theorems \ref{resultat2intro} and \ref{dimcourantfinieintro}). Next we display our result concerning the general closed positive currents $S$.

\subsubsection{Directional dimensions of the Green current}\label{G}

The invariance property of the current $T$ allows to obtain estimates on the directional dimensions  $\nu$-almost everywhere. In what follows, the functions $O(\epsilon)$ only depend on $\epsilon$, the degree $d$ of the endomorphism, the exponents and the entropy of $\nu$. They tend to $0$ when  $\epsilon$ tends to $0$ and are positive for  $\epsilon$ small enough. We shall say that the exponents of $\nu$ do not resonate if $\lambda_1 \neq k \lambda_2$  for every $k \geq 2$.

\begin{thm}\label{resultat2intro} Let $f$ be an endomorphism of $\bbp^2$ of degree $d \geq 2$. Let  $\nu$ be an ergodic diltating measure whose exponents $\lambda_1 > \lambda_2$ do not resonate. Then there exist functions $O_1(\epsilon), O_2(\epsilon)$ satisfying the following properties. For every $\epsilon >0$ and for $\nu$-almost every $x$, there exist holomorphic coordinates 
 $(Z, W)$ in a neighbourhood of  $x$ such that
$$\bar{d_{T,Z}}(x) \geq 2 + \frac{h_{\nu} - \log d}{\lambda_2} - O_1(\epsilon),$$
$$\bar{d_{T,W}}(x) \geq 2 \frac{\lambda_2}{\lambda_1} + \frac{h_{\nu} - \log d}{\lambda_2} - O_2(\epsilon).$$
\end{thm}

This result applies to the measure $\mu$ and allows to improve a classical lower bound  concerning the dimension of the Green current $T$. This current  has local $\gamma$-H\"older \emph{psh} potentials for every $\gamma < \gamma_0:= \min \{ 1 , \log d / \log d_\infty \}$, where $d_\infty:= \lim_n \vert \vert Df^n \vert \vert_{\infty}^{1/n}$, see \cite[Proposition 1.18]{dinsib10}. This implies that $T \wedge \omega (B_x(r)) \leq c_\gamma(x) r^{2+\gamma}$ for every $x \in \bbp^2$ and every $\gamma < \gamma_0$, see \cite[Th\'eor\`eme 1.7.3]{sib99}. We deduce that for every $x \in \bbp^2$ and for every local holomorphic coordinates $(Z,W)$ in a neighbourhood of $x$:
\begin{equation}\label{hol}
\min \{ \bar{ d_{T,Z}}(x) , \bar{d_{T,W}}(x)  \}   = \bar{d_T}(x)  \geq 2 + \gamma_0 . 
\end{equation}
Now, since $h_\mu = \log d^2$ and $\lambda_1 \leq \log d_\infty$, we get:
$$ {h_\mu - \log d \over \lambda_2  }  = {\log d \over \lambda_2  } > {\log d \over \lambda_1  }  \geq {\log d \over \log d_\infty} \geq \gamma_0 . $$
The lower bound in the direction $Z$ provided by Theorem  \ref{resultat2intro} is therefore better than (\ref{hol}) when $\epsilon$ is small enough. \\

Now we show directional upper bounds for the current $T$ with respect to every dilating measure $\nu$ whose support is contained in the support of  $\mu$.

\begin{thm} \label{dimcourantfinieintro} 
Let $f$ be an endomorphism of $\bbp^2$ of degree $d \geq 2$. Let  $\nu$ be an ergodic dilating measure whose exponents $\lambda_1 > \lambda_2$ do not resonate. We assume that the support of $\nu$ is contained in the support of $\mu$. Then there exist functions $O_3(\epsilon), O_4(\epsilon)$ satisfying the following properties. For every $\epsilon >0$ and for $\nu$-almost every $x$, there exist holomorphic coordinates $(Z, W)$ in a neighbourhood of  $x$ such that
$$ \underline{d_{T,Z}}(x) \leq \frac{\log d}{\lambda_2} + 2 \frac{\lambda_1}{\lambda_2} + O_3(\epsilon),$$ 
$$\underline{d_{T,W}}(x) \leq \frac{\log d}{\lambda_2} + 2 + O_4(\epsilon). $$
\end{thm}

By combining these two theorems, we manage to separate coordinates $(Z,W)$ in terms of local dimensions.

\begin{cor}\label{coroo} Let $f$ be an endomorphism of $\bbp^2$ of degree $d \geq 2$. Let  $\nu$ be an ergodic dilating measure whose exponents $\lambda_1 > \lambda_2$ do not resonate. For every $\epsilon >0$ and for $\nu$-almost every $x$, there exist holomorphic coordinates $(Z, W)$ in a neighbourhood of  $x$ such that
$$  \underline{d_{T,W}}(x)  \leq \frac{\log d}{\lambda_2} + 2 + O_4(\epsilon)   \ \ , \ \    \frac{\log d}{\lambda_2} + 2 - O_1(\epsilon)\leq  \bar{d_{T,Z}}(x)  . $$ 
\end{cor}

In the three preceding results, the coordinates $(Z,W)$ come from a normal form Theorem for the inverse branches of $f^n$, see Section \ref{SOP}. They depend on $\hat x$ in the natural extension $(\hat f, \hat \nu)$, we shall denote them by $(Z_{\hat{x}}^{\epsilon},W_{\hat{x}}^{\epsilon})$. The coordinates  $(Z,W)$ at a point $x$ are coordinates of type $(Z_{\hat{x}}^{\epsilon},W_{\hat{x}}^{\epsilon})$ where $\pi_0(\hat x) = x$. The coordinate $W_{\hat{x}}^{\epsilon}$ is always invariant by the shift $\hat f$, $Z_{\hat{x}}^{\epsilon}$ is invariant when the exponents do not resonate. Since the current $T$ is $f$-invariant, the functions $\underline {d_{T,Z}}(\hat x)$ and $\underline {d_{T,W}}(\hat x)$ are $\hat f$-invariant, hence  $\hat \nu$-almost everywhere constant, see Proposition \ref{propinvdimT2}. We shall denote them by $\underline {d_{T,Z}}(\nu)$ and $\underline {d_{T,W}}(\nu)$. We have the same properties for the upper dimensions. 

\subsubsection{Dimension of ergodic currents $S$}

\begin{thm}\label{thm1intro} Let $f$ be an endomorphism of $\bbp^2$ of degree $d \geq 2$. Let $S$ be a $(1,1)$-closed positive current on $\bbp^2$. We assume that the support of $S$ contains an ergodic measure  $\nu$ of entropy $h_\nu > \log d$ whose exponents satisfy $\lambda_1 > \lambda_2$ and do not resonate. Then there exists a function $O_5(\epsilon)$ satisfying the following properties. For every  $\epsilon >0$, there exist $x \in \Sup \nu$ and a holomorphic coordinate  $Z$ in the neighbourhood of $x$ such that:
$$ \bar{d_{S,Z}}(x) \geq 2 + \frac{h_{\nu} - \log d}{\lambda_2} - O_5(\epsilon).$$
In particular, $S$ has a local directional dimension  $>2$ at some $x \in \Sup \nu$.
\end{thm}

This result is localized at a point  $x$ because $S$ is not assumed to be $f$-invariant. For every closed positive current $S$ containing an ergodic dilating measure $\nu$ with exponents $\lambda_1 \geq \lambda_2$, de Th\'elin-Vigny proved in \cite{detvig15} that there exists  $x \in supp (\nu)$ such that
 \begin{equation}\label{dtvdup} \bar{d_S}(x) \geq 2 \frac{\lambda_2}{\lambda_1} + \frac{h_{\nu} - \log d}{\lambda_2} .
\end{equation}
Theorem \ref{thm1intro}  improves this estimate by replacing $\lambda_2/ \lambda_1$ by $1$ for a coordinate $Z$. When  $\lambda_1 = \lambda_2$, (\ref{minco}) shows that this substitution is valid for every coordinate $Z$. Our lower bound answers a question of  \cite{detvig15} in a directional way. In the framework of invertible and meromorphic mappings, the preprint  \cite{det17} gives a lower bound $>2$ for the dimensions of the Green currents $T^\pm$ by using coding techniques and laminar properties of $T^\pm$. 

\subsection{Dimension of dilating measures}\label{DMD}

Let $f$ be an endomorphism of $\bbp^2$ of degree $d \geq 2$ and let $\nu$ be an ergodic measure. We refer to \cite{pes97} and \cite{you82} for the beginning of this Section. The dimension of $\nu$ is defined by  
$$\dim_H(\nu):= \inf \set{\dim_H(A) \, , \, A \textrm{ Borel set of } \bbp^2 \, , \, \nu(A) = 1} .$$
The lower and upper local dimensions of  $\nu$ are
$$ \underline{d_{\nu}}:= \liminf_{r  \to 0}  \frac{\log (\nu(B_x(r)))}{\log r} \ \ , \ \ \bar{d_{\nu}}:= \limsup_{r  \to 0}  \frac{\log (\nu(B_x(r)))}{\log r} , $$
these limits are $\nu$-almost everywhere constant, by ergodicity of $\nu$. If $a \leq \underline{d_{\nu}} \leq \bar{d_{\nu}} \leq b$, then $a \leq \dim_H(\nu)  \leq b$. If $\nu$ is dilating, we have the classical inequalities $\frac{h_{\nu}}{\lambda_1} \leq \underline{d_{\nu}} \leq \bar{d_{\nu}} \leq \frac{h_{\nu}}{\lambda_2}$ where  $\lambda_1 \geq \lambda_2$ are the exponents of $\nu$.
For the equilibrium measure $\mu$, Binder-DeMarco \cite{bindem03} conjectured the formula:
\begin{equation}\label{BDM}
\dim_H(\mu) = \frac{\log d}{\lambda_1} + \frac{\log d}{\lambda_2}
\end{equation}
which generalizes the one-dimensional Ma\~n\'e's formula \cite{man88}. The article \cite{bindem03} proves the following upper bound for polynomial mappings
\begin{equation}\label{bdmdd}
\dim_H(\mu) \leq 4 - {  2(\lambda_1 + \lambda_2) - \log {d^2}  \over \lambda_1 } ,
\end{equation}
the article \cite{dindup04} extends this upper bound in a meromorphic context. For every dilating measure  $\nu$, we have at our disposal the following lower bound proved in  \cite{dup11}:
\begin{equation}\label{dupdup}
\dim_H ( \nu ) \geq \underline{d_{\nu}} \geq  \frac{\log d}{\lambda_1} + \frac{h_{\nu} - \log d}{\lambda_2}   .
\end{equation}
We obtain a new upper bound which comes closer to Binder-DeMarco's conjecture (\ref{BDM}). 

\begin{thm} \label{resultat3}  Let $f$ be an endomorphism of $\bbp^2$ of degree $d \geq 2$. Let $\nu$ be an ergodic dilating measure, of exponents $\lambda_1 > \lambda_2$ and whose support is contained in the support of  $\mu$. Then
$$ \underline{d_{\nu}} \leq \frac{\log d}{\lambda_1} + \frac{\log d}{\lambda_2} + 2 \left(  1 - \frac{\lambda_2}{\lambda_1} \right).  $$
Moreover, if the exponents do not resonate, then: 
$$\underline{d_{\nu}} \leq \frac{\log d}{\lambda_1} + \frac{\log d}{\lambda_2} +  2 \min  \left(  1 - \frac{\lambda_2}{\lambda_1}  ; \frac{\lambda_1}{\lambda_2}-1 \right) . $$
\end{thm}
The proof extends the arguments of Theorem \ref{dimcourantfinieintro}. 

\subsection{Dimension of semi-extremal endomorphisms}\label{SE}
 
We say that $f$ is \emph{extremal} if the exponents $\lambda_1 \geq \lambda_2$ of its equilibrium measure $\mu$ satisfy $\lambda_1 = \lambda_2 = \frac{1}{2} \log d$. 
 The articles \cite{berdup05, BL01, dindup04} characterize these endomorphisms by the equivalent properties:
\begin{enumerate}
\item $\dim_H(\mu) = 4$.
\item $\mu << Leb_{\bbp^2}:= \omega \wedge \omega$.
\item $T$ is a $(1,1)$ positive smooth form on an open set of $\bbp^2$.
\item $f$ is a  Latt\`es map: there exist a complex torus $\bbc^2 / \Lambda$, an affine dilation $D$ on this torus and a finite galoisian covering $\sigma : \bbc^2 / \Lambda \to \bbp^2$ such that the following diagram commutes 
\[ \xymatrix{
 \bbc^2 / \Lambda \ar[d]_{\sigma} \ar[r]^{D} &   \bbc^2 / \Lambda  \ar[d]_{\sigma}  \\
  \bbp^2 \ar[r]^{f}   &   \bbp^2 } \]
\end{enumerate}
There exist such applications  for every degree $d \geq 2$, see for instance \cite{dup03}. 
 
We say that $f$ is \emph{semi-extremal} if the exponents  $\lambda_1 \geq \lambda_2$ of its equilibrium measure $\mu$ satisfy $\lambda_1 > \lambda_2 = \frac{1}{2} \log d$. Using (\ref{bdmdd}) and (\ref{dupdup}) one sees that the Binder-DeMarco's conjecture (\ref{BDM}) holds for these mappings: 
\begin{equation} \label{corollairedimintro}
\dim_H (\mu ) = 2 + \frac{\log d}{\lambda_1} .
\end{equation}
Classical examples of semi-extremal endomorphisms are  suspensions of one-dimensional Latt\`es maps, they satisfy $\mu << T \wedge \omega$. More generally,    

\begin{thm}[Dujardin \cite{Duj12}] \label{dujardin}
If $\mu << T \wedge \omega$, then $f$ is semi-extremal.
\end{thm}

In \cite{Duj12} Dujardin asked if $\mu << T \wedge \omega$ implies the existence of a one-dimensional Latt\`es factor for  $f$. Theorem  \ref{thmAintro} below provides one step in that direction. It shows that, from a theoritical dimensional point of view, these endomorphisms look like suspensions of one-dimensional Latt\`es maps: the dimension of $T$ is  maximal equal to $4$ in a coordinate $Z$, and equal to $2 + \log d / \lambda_1$ (the dimension of $\mu$, see (\ref{corollairedimintro})) in a coordinate $W$. The  functions $O_1(\epsilon), O_2(\epsilon)$ come from Theorem \ref{resultat2intro}.

 \begin{thm} \label{thmAintro}
Let $f$ be an endomorphism of  $\bbp^2$ of degree $d \geq 2$. We assume that $\mu << T \wedge \omega$ and that $\bar{d_{\mu}} = \underline{d_{\mu}}$. We also assume that the exponents $\lambda_1 > \lambda_2 = \frac{1}{2} \log d$ of $\mu$ do not resonate. For every $\epsilon >0$ and  for $\mu$-almost every $x \in \bbp^2$,  there exist holomorphic coordinates $(Z,W)$ in a neighbourhoord of $x$ such that
$$ 4 - O_1(\epsilon)\leq  \bar{d_{T,Z}}(x) \ \ \textrm{ and  } \ \ \ 2 + \frac{\log d}{\lambda_1} -  O_2(\epsilon) \leq \bar{d_{T,W}}(x) \leq  2 + \frac{\log d}{\lambda_1}.$$
\end{thm}

\subsection{Organization of the article}

Sections \ref{SOP}, \ref{GBBU} et  \ref{ES} are devoted to normal forms, the geometry of inverse branches and  separated sets. Sections \ref{MIDD} et \ref{ECG} establish  Theorems \ref{resultat2intro}, \ref{thm1intro} and  \ref{thmAintro}, the proofs rest on  Theorem \ref{resultat2} which relies on the lower local dimension $\underline{d_{\nu}}$ (the lower bound (\ref{dupdup}) is therefore crucial to deduce these results). Theorems \ref{dimcourantfinieintro} and \ref{resultat3} are proved in Sections \ref{MDD} and \ref{MDMD}. Section \ref{appendice} brings together technical results.\\

{\bf {Acknowledgements: }} This article is part of the PhD thesis of the second author. We warmly thank Eric Bedford and Johan Taflin for their numerous comments which allow to improve the first version of this work. We got the supports of Lambda (ANR-13-BS01-0002) and Centre Henri Lebesgue (ANR-11-LABX-0020-01).

\section{Normal forms and Oseledec-Poincar\'e coordinates}\label{SOP}

\subsection{Natural extension and normal forms}

Let $\mathcal{C}_f$ be the critical set of  $f$,  this is an algebraic subset of $\bbp^2$. If $\nu$ is an ergodic dilating measure, then $x \mapsto \log \abs{ Jac \,  f(x)} \in L^1(\nu)$, which implies $\nu(\mathcal{C}_f) =0$. Let $X$ be the $f$-invariant Borel set $\bbp^2 \setminus \cup_{n \in \bbz} f^n (\mathcal{C}_f) $ and let 
$$\hat{X}: =  \set{ \hat{x}  = (x_n)_{n \in \bbz} \in X^{\bbz}, ~~  x_{n+1} = f (x_{n}) } . $$
Let $\hat f$ be the left shift on $\hat X$ and $\pi_0 (\hat x) : = x_0$. There exists a unique  $\hat f$-invariant  measure $\hat{\nu}$ on $\hat{X}$ such that $(\pi_0)_* \hat \nu = \nu$. We set $\hat{x}_n := \hat{f}^n(\hat{x})$ for $n \in \bbz$. A function $\alpha : \hat{X} \to \left] 0,+ \infty \right] $ is \emph{$\epsilon$-tempered} if $\alpha(\hat{f}^{\pm 1}(\hat{x})) \geq e^{- \epsilon} \alpha(\hat{x})$. For every $\hat x \in X$ we denote by $f^{-n}_{\hat{x}}$ the inverse  branch of  $f^n$  defined in a neighbourhood of  $x_0$ with values  in a neighbourhood of $x_{-n}$. The articles \cite{bdm07} and \cite{jonvar02} provide normal forms for these mappings.

\begin{thm} (\cite[Proposition 4.3]{bdm07}) \label{FN1} Let $f$ be an endomorphism of $\bbp^2$ of degree $d \geq 2$. Let $\nu$ be an ergodic dilating measure with exponents $\lambda_1 > \lambda_2$. Let $\epsilon > 0$. 

There exists an $\hat{f}$-invariant Borel set $\hat{F} \subset \hat{X}$ such that $\hat{\nu}(\hat{F}) = 1$ and satisfying the following properties. There exist $\epsilon$-tempered   functions $\eta_\epsilon, {\rho}_\epsilon : \hat{F} \to \left] 0,1 \right]$ and $\beta_\epsilon, L_\epsilon, M_{\epsilon} : \hat{F} \to [1,+\infty[$ and for every $\hat x \in \hat{F}$, there exists an injective holomorphic mapping 
$$\xi_{\hat x}^{\epsilon} : B_{x_0}(\eta_{\epsilon}(\hat x)) \to \bbd^2(\rho_{\epsilon}(\hat x))$$
such that the following diagram commutes for every $n \geq n_\epsilon(\hat{x})$:
\[ \xymatrix{
  B_{x_{-n}}(\eta_{\epsilon}(\hat{x}_{-n})) \ar[d]^{\xi_{\hat{x}_{-n}}^{\epsilon}}  & &  B_{x_0}(\eta_\epsilon (\hat{x})) \ar[ll]_{f^{-n}_{\hat x}} \ar[d]_{\xi_{\hat x}^{\epsilon}}  \\
  \bbd^2(\rho_{\epsilon} (\hat{x}_{-n}))  & & \bbd^2(\rho_{\epsilon}(\hat{x}))  \ar[ll]^{R_{n,\hat{x}}} }  \]
and such that
\begin{enumerate}
\item $\forall (p,q) \in   B_{x_0}(\eta_{\epsilon}(\hat{x})),~ \frac{1}{2} d(p,q) \leq \abs{\xi_{\hat{x}}^{\epsilon}(p) - \xi_{\hat{x}}^{\epsilon}(q)} \leq \beta_{\epsilon}(\hat{x}) d(p,q)$.
\item $\Lip (f_{\hat x}^{-n}) \leq L_\epsilon(\hat x) e^{-n \lambda_2 + n \epsilon}$ on $B_{x_0}(\eta_\epsilon(\hat x))$.

\item  if $\lambda_1 \not\in \set{k \lambda_2, k \geq 2}$, 
$R_{n,\hat{x}} (z,w)  = (\alpha_{n,\hat{x}} z, \beta_{n,\hat{x}} w)$, 

if  $\lambda_1 = k \lambda_2$ where $k \geq 2$,
 $R_{n,\hat{x}} (z,w) = (\alpha_{n,\hat{x}} z, \beta_{n,\hat{x}} w) + ( \gamma_{n,\hat{x}} w^k , 0)$, with  
 
\begin{enumerate}
\item $ e^{- n \lambda_1 - n \epsilon } \leq \abs{\alpha_{n, \hat{x}}} \leq e^{- n \lambda_1 + n \epsilon}$ and $\abs{\gamma_{n, \hat{x}}} \leq M_{\epsilon}(\hat{x}) e^{- n \lambda_1 + n \epsilon} $,
\item $e^{- n \lambda_2 - n \epsilon} \leq \abs{\beta_{n, \hat{x}}} \leq  e^{- n \lambda_2 + n \epsilon}$. 
\end{enumerate}
\end{enumerate}
\end{thm} 

\begin{rem}\label{req}
The diagram commutes  for every $n \in \set{1, \dots , n_\epsilon(\hat{x})}$ for the germs of the mappings, see \cite{bdm07}. The integer $n_\epsilon(\hat x)$ is the smallest integer such that $L_\epsilon(\hat x ) e^{-n\lambda_2 + n \epsilon} \leq e^{-n\epsilon}$, so that $L_\epsilon(\hat x ) e^{-n\lambda_2 + n \epsilon} \eta_\epsilon (\hat x) \leq  e^{-n\epsilon} \eta_\epsilon(\hat x) \leq \eta_\epsilon(\hat x_{-n})$. Item 2 thus ensures  that $f^{-n}_{\hat x} (B_{x_0}(\eta_\epsilon(\hat x)) \subset   B_{x_{-n}}(\eta_{\epsilon}(\hat{x}_{-n}))$. 
\end{rem}

We shall need the following Lemma. Let $n_1(L)$ be the smallest integer  $n$ satisfying $L / 4 \leq e^{n\epsilon}$.
The first item uses the upper bound for $\Lip (f_{\hat x}^{-n})$ provided by Theorem \ref{FN1}. The second item comes from  \cite[Proposition 3.1]{dindup04}. 

\begin{lemme}  \label{lemmeinclusion2}
Let $\hat{x} \in \hat{F}$ such that  $\eta_{\epsilon}(\hat{x}) \geq \eta$ and $L_{\epsilon}(\hat{x}) \leq L$. If $n \geq n_1(L)$ and $r \leq \eta$, 
\begin{enumerate}
\item $f^{-n}_{\hat{x}_n}(B_{x_n}(r/4)) \subset B_{x_0}(r e^{- n \lambda_2 + 3 n \epsilon})$ and $f^{-n}_{\hat{x}}(B_{x_0}(r/4)) \subset B_{x_{-n}}(r e^{- n \lambda_2 + 3 n \epsilon})$.
\item $ B_{x_0}(r e^{- n \lambda_1 - 4 n \epsilon}) \subset f^{-n}_{\hat{x}_n}(B_{x_n}(\frac{r}{4} e^{- 2 n \epsilon}))$.
\end{enumerate}
\end{lemme}

\subsection{Oseledec-Poincar\'e coordinates}\label{coOP}

Let $\nu$ be an ergodic dilating measure of exponents $\lambda_1 > \lambda_2$. We assume that the exponents do not resonate, which means that $\lambda_1 \not\in \set{k \lambda_2, k \geq 2}$. Let $\epsilon > 0$, and let us apply Theorem \ref{FN1}. For every  $\hat x \in \hat{F}$ we denote by $(Z_{\hat{x}}^{\epsilon},W_{\hat{x}}^{\epsilon})$ the coordinates of $\xi_{\hat{x}}^{\epsilon}$. The commutative diagram given by Theorem \ref{FN1} implies: 
\begin{equation} \label{eq033}
Z^{\epsilon}_{\hat{x}_{-n}} \circ f^{-n}_{\hat{x}} = \alpha_{n, \hat{x}} \times Z^{\epsilon}_{\hat{x}}, \quad W^{\epsilon}_{\hat{x}_{-n}} \circ f^{-n}_{\hat{x}} = \beta_{n, \hat{x}} \times W^{\epsilon}_{\hat{x}}.
\end{equation}
Hence, $f^{-n}_{\hat{x}}$ multiplies the first coordinate by $e^{-n \lambda_1 \pm n \epsilon}$ and multiplies the second coordinate by $e^{-n \lambda_2 \pm n \epsilon}$. Let us note that the second property holds in the resonant case $\lambda_1 \in \{ k  \lambda_2 , k \geq 2\}$. We shall name the collection of local holomorphic coordinates  
$$(Z,W)_\epsilon := (Z_{\hat{x}}^{\epsilon},W_{\hat{x}}^{\epsilon})_{\hat x \in \hat{F}}  $$
 \emph{Oseledec-Poincar\'e coordinates for $(f,\nu)$}. Using  (\ref{eq033}) and the fact that the Green current is $f$-invariant, we obtain the following Proposition.
 
\begin{propo} \label{propinvdimT2} Let $f$ be an endomorphism of $\bbp^2$ of degree $d \geq 2$ and let $T$ be its Green current. Let $\nu$ be an ergodic dilating measure of exponents $\lambda_1 > \lambda_2$. Let $(Z,W)_\epsilon$ be Oseledec-Poincar\'e coordinates for $(f,\nu)$. Then there exists a $\hat f$-invariant Borel set $\hat{\Lambda}_T \subset \hat{F}$ of $\hat{\nu}$-mesure 1 such that
\begin{enumerate}
\item  $\hat{x} \mapsto \bar{d_{T,W}}(\hat{x})$ and $\hat{x} \mapsto \underline{d_{T,W}}(\hat{x})$ are $\hat f$-invariant on $\hat{\Lambda}_T$.
\item  $\hat{x} \mapsto \bar{d_{T,Z}}(\hat{x})$ and $\hat{x} \mapsto \underline{d_{T,Z}}(\hat{x})$ are $\hat f$-invariant on $\hat{\Lambda}_T$ if $\lambda_1 \not\in \set{k \lambda_2, k \geq 2}$.
\end{enumerate}
In particular, if the exponents do not resonate, these functions are constant $\hat{\nu}$-almost everywhere. We shall denote them by
$$\bar{d_{T,Z}}(\nu),~ \underline{d_{T,Z}}(\nu),~ \bar{d_{T,W}}(\nu),~ \underline{d_{T,W}}(\nu) . $$
\end{propo}

\begin{proof}
We prove the invariance of $\bar{d_{T,W}}(\hat{x})$, the same arguments hold for the other functions.
For every $z \in \bbp^2 \setminus \mathcal{C}_f$ we denote
$$ a(z) := {1 \over 2} \vert \vert (D_z f)^{-1} \vert \vert ^{-1} \ , \  \gamma(z) : = \min \{  a(z) \norme{f}_{\mathcal{C}^2, \bbp^2}^{-1} \, , \,  1 \} . $$
Then \cite[Lemme 2]{BriDuv99} asserts that $f$ is injective on $B_z(\gamma(z))$ and
$$ \forall  r \in \left[ 0, \gamma(z) \right] \ , \ B_{f(z)} \left( a(z) r \right) \subset f(B_z(r)) . $$
Let $\hat{x} \in \hat{F}$. Since $x_n \not\in \mathcal{C}_f$ for every $n \in \bbz$, we obtain for every $r \leq \gamma(x_0)$:
 \begin{equation}\label{ko1}
 T \wedge \left( \frac{i}{2} d W_{\hat{f}(\hat{x})}^{\epsilon} \wedge d \bar{W_{\hat{f}(\hat{x})}^{\epsilon}} \right) \left[ B_{f(x_0)} \left(  a(x_0) r \right) \right] \leq  T \wedge \left( \frac{i}{2} d W_{\hat{f}(\hat{x})}^{\epsilon} \wedge d \bar{W_{\hat{f}(\hat{x})}^{\epsilon}} \right)\left[ f(B_{x_0}(r)) \right] . 
  \end{equation}
Since $f$ is injective on $B_{x_0}(r)$, we can change the variables to get:
 \begin{equation}\label{ko2}
 T \wedge \left( \frac{i}{2} d W_{\hat{f}(\hat{x})}^{\epsilon} \wedge d \bar{W_{\hat{f}(\hat{x})}^{\epsilon}} \right)\left[ f(B_{x_0}(r)) \right]   = \int_{B_{x_0}(r)} f^{*}T \wedge  f^* \left( \frac{i}{2} d W_{\hat{f}(\hat{x})}^{\epsilon} \wedge d \bar{W_{\hat{f}(\hat{x})}^{\epsilon}} \right) . 
\end{equation}
Now let us recall that   
\begin{equation}\label{ko3}
f^*T = d T \ \textrm { and } \ f^* \left( \frac{i}{2} d W_{\hat{f}(\hat{x})}^{\epsilon} \wedge d \bar{W_{\hat{f}(\hat{x})}^{\epsilon}} \right) = \vert c(\hat{x}) \vert ^2 \frac{i}{2} d W_{\hat{x}}^{\epsilon} \wedge d \bar{W_{\hat{x}}^{\epsilon}} ,
\end{equation}
where the second equality comes from (\ref{eq033}) by setting $c(\hat{x})^{-1} := \beta_{1,\hat{f}(\hat{x})}$, it is valid near $x_0$ according to Remark \ref{req}. By combining (\ref{ko1}), (\ref{ko2}) and (\ref{ko3}) we deduce:
$$ T \wedge \left( \frac{i}{2} d W_{\hat{f}(\hat{x})}^{\epsilon} \wedge d \bar{W_{\hat{f}(\hat{x})}^{\epsilon}} \right) \left[ B_{f(x_0)} \left( a(x_0) r \right) \right] \leq d  \vert c(\hat{x}) \vert ^2 ~T \wedge \left( \frac{i}{2} d W_{\hat{x}}^{\epsilon} \wedge d \bar{W_{\hat{x}}^{\epsilon}} \right)\left[ B_{x_0}(r) \right]  $$
for every $r$ small enough. Taking the logarithm and dividing by $\log ( a(x_0) r) < 0$, we get 
$\bar{d_{T,W}}(\hat{f}(\hat{x}))\geq  \bar{d_{T,W}}(\hat{x})$ by taking limits. Since $\hat{\nu}$ is ergodic, the function $\bar{d_{T,W}}(\hat{x})$ is constant on a Borel set $\hat{\Lambda}_T$ of $\hat{\nu}$-measure 1 (see \cite[Chapter 1.5]{Walters82}). One can replace it by $\bigcap_{n \in \bbz} \hat{f}^n ( \hat{\Lambda}_T)$ to obtain an invariant set.
 \end{proof}

\begin{propo} \label{ergodim} Let $f$ be an endomorphism of $\bbp^2$ and let $T$ be its Green current. Let $\nu$ be an ergodic measure. Then the functions $x \mapsto \underline{d_T}(x)$ and $x \mapsto \bar{d_T}(x)$ are invariant, hence $\nu$-almost everywhere constant. We denote them by $\underline{d_T}(\nu)$ and $\bar{d_T}(\nu)$.
\end{propo}

\begin{proof}
The arguments follow the proof of Proposition \ref{propinvdimT2}. In this case we study the measure $T \wedge \omega$, and we replace the second equality in (\ref{ko3}) by $f^* \omega \leq \rho(x_0) \omega$ on $B_{x_0}(\gamma(x_0))$, where $\rho(x_0) > 0$ is a large enough positive constant.
\end{proof}

\section{Geometry of the inverse branches  and  uniformizations}\label{GBBU}

Let $\nu$ be an ergodic dilating measure of exponents $\lambda_1 > \lambda_2$. Let $\epsilon > 0$ and let $(Z,W)_\epsilon$ be Oseledec-Poincar\'e coordinates for $(f,\nu)$. Our aim is to construct, for every $\delta > 0$, a Borel set $\hat{\Lambda}_{\epsilon} \subset \hat{X}$ satisfying $\hat \nu(\hat{\Lambda}_{\epsilon}) \geq 1 - \delta /2$  which provides convenient uniformizations.

\subsection{Dynamical balls}

The dynamical distance is defined by $d_n(x,y) := \max_{0 \leq k \leq n} d(f^k(x),f^k(y))$. We denote by $B_n(x,r)$ the ball centered at $x$ and of radius $r$ for the distance $d_n$.

\begin{lem} \label{BK} There exist $r_0 >0$, $n_2 \geq 1$ and  $C \subset \bbp^2$ such that $\nu(C) \geq 1-\delta/8$ and satisfying the following properties: for every $x \in C$ and every $n \geq n_2$:
 \[ \nu (B_n(x,r_0/8)) \geq e^{-n h_{\nu} - \epsilon n}.\]
\[\forall r \leq r_0 \ , \  \nu (B_n(x,5r)) \leq \nu (B_n(x,5 r_0)) \leq e^{-n h_{\nu} + \epsilon n} . \]
\end{lem}

\begin{proof}
Brin-Katok Theorem  \cite{brikat83} ensures that there exists $C_1 \subset \bbp^2$ of full $\nu$-measure such that for every $x \in C_1$:
\[   \lim_{r \to 0} \left( \liminf_{n \to + \infty} \frac{-1}{n} \log \nu (B_n(x,r)) \right) = \lim_{r \to 0} \left( \limsup_{n \to + \infty} \frac{-1}{n} \log \nu (B_n(x,r)) \right) = h_{\nu}. \]
Hence, for every $x \in C_1$ there exists $r_0(x)$ such that $r \leq r_0(x)$ implies
 \[ \liminf_{n \to + \infty} \frac{-1}{n} \log (\nu (B_n(x,5r))) \geq h_{\nu} - \epsilon/2 \text{ and } \limsup_{n \to + \infty} \frac{-1}{n} \log (\nu (B_n(x,r/8))) \leq h_{\nu} + \epsilon/2 . \]
Let $r_0$ such that  $C_2 := \set{ x \in C_1 \, , \,  r_0(x) \geq r_0}$ satisfies $\nu( C_2 ) \geq 1 - \delta/16$. For every $x \in C_2$, there exists $n_2(x)$ such that $n \geq n_2(x)$ implies 
\[ \nu (B_n(x,r_0/8)) \geq e^{-n h_{\nu} - \epsilon n}.\]
\[\forall r \leq r_0 \ , \  \nu (B_n(x,5r)) \leq \nu (B_n(x,5 r_0)) \leq e^{-n h_{\nu} + \epsilon n} . \]
Let $n_2 \geq 1$ such that $C := C_2 \cap \{ x \in C_1 \, , \, n_2(x) \leq n_2 \}$ satisfies $\nu(C) \geq 1-\delta/8$.
\end{proof}

For every $L>0$, let $m_L$ be the smallest integer $m$ such that $L e^{-m (\lambda_2 + \epsilon)} \leq 1$ and let $n_3(L)$ be the smallest integer larger than $m_L$ such that $e^{- n \epsilon} \leq M^{- m_L}$, where $M := \max \{ \norme{D f}_{\infty, \bbp^2} , 1 \}$.

\begin{lemme}
\label{lemmeinclusion1}
Let $\hat{x} \in \hat{F}$ such that  $\eta_{\epsilon}(\hat{x}) \geq \eta$ and $L_{\epsilon}(\hat{x}) \leq L$. For every  $n \geq n_3(L)$ and every $r \leq \eta$,
$$ f^{-n}_{\hat{x}_n}(B_{x_n}(r e^{- 2 n \epsilon})) \subset B_n(x_0,r) . $$
\end{lemme}

\begin{proof} Let us observe that for every $0 \leq k \leq n$, $f^k$ is injective on $f^{-n}_{\hat{x}_n}(B_{x_n}(r e^{- 2 n \epsilon}))$ and that  $f^k f^{-n}_{\hat{x}_n} = f^{-n+k}_{\hat{x}_n}$. By setting $ p = n-k$, it suffices to show that
\begin{equation}\label{eq013}
\forall p \in \ensent{0,n} \quad f^{-p}_{\hat{x}_n}(B_{x_n}(r e^{- 2 n \epsilon})) \subset B_{x_{n-p}}(r) .
\end{equation}
To simplify let us set $m := m_L$ et $n_3 := n_3(L)$. We immediately have
\begin{equation} \label{eq017}
\forall n \geq n_3, ~ \forall p \in \ensent{0,n}, \quad f_{\hat{x}_n}^{-p}(B_{x_n}(r e^{- 2 n \epsilon})) \subset  f_{\hat{x}_n}^{-p}(B_{x_n}(\frac{r}{M^m} e^{- n \epsilon})).
\end{equation}
To verify (\ref{eq013}), we shall consider separately the cases $p \leq m$ and $p > m$. We know that for every $p$, $\Lip f^{-p}_{\hat{x}_n} \leq L(\hat{x}_n) e^{- p \lambda_2 + p \epsilon} \leq  L e^{n \epsilon} e^{- p \lambda_2 + p \epsilon} $ on $B_{x_n} ( \eta_{\epsilon} (\hat{x}_n)  )$, which contains $B_{x_n} ( \eta e^{- n \epsilon}  )$.
Hence for every $n \geq n_3 \geq m$,  $p \in \ensent{m,n}$ and $r \leq \eta$:
$$f^{-p}_{\hat{x}_n} (B_{x_n}(r e^{-n \epsilon})) \subset B_{x_{n-p}} (r e^{-n \epsilon} Le^{n \epsilon}e^{-p \lambda_2 + p \epsilon}) = B_{x_{n-p}} (r L e^{-p \lambda_2 + p \epsilon}) \subset B_{x_{n-p}}(r).$$
Since $M^m \geq 1$ this implies for every $n \geq n_3 \geq m$, $p \in \ensent{m,n}$ and $r \leq \eta$:
\begin{equation} \label{eq010} 
f^{-p}_{\hat{x}_n} ( B_{x_n} ( \frac{r}{M^m} e^{ - n \epsilon} )  ) \subset B_{x_{n-p}} ( \frac{r}{M^m})  .
\end{equation}
Thus, by using (\ref{eq017}) and $M^m \geq 1$:
\begin{equation} \label{eq011} 
\forall p \in \ensent{m,n}, \quad f^{-p}_{\hat{x}_n} ( B_{x_n} ( r e^{ -  2 n \epsilon} )  ) \subset B_{x_{n-p}} (r)  .
\end{equation}
We have proved (\ref{eq013}) for $p \in \ensent{m,n}$. Let us show this inclusion for $ p \in \ensent{0,m}$. For every $p \in \ensent{0,m}$, let us set $p = m - p'$ where $p' \in \ensent{0,m}$. Then
\begin{align*}
 f^{-p }_{\hat{x}_n} (B_{x_n}(\frac{r}{M^m} e^{-n \epsilon}))  = f^{p' } (f^{-m}_{\hat{x}_n} (B_{x_n}(\frac{r}{M^m} e^{-n \epsilon}))  \subset  f^{p'} (B_{x_{n-m}}(\frac{r}{M^m} ))  ,
\end{align*}
where the inclusion comes from (\ref{eq010}) with $p=m$. We deduce:
\begin{equation} \label{eq012} 
 \forall p \in \ensent{0,m}, \quad  f^{-p }_{\hat{x}_n} (B_{x_n}(r e^{- 2 n \epsilon})) \subset B_{x_{n-m+p'}}(\frac{r}{M^m}M^p ) \subset B_{x_{n-p}} (r) .
\end{equation}
We finally obtain (\ref{eq013}) by combining (\ref{eq011}) and (\ref{eq012}).
\end{proof}

\subsection{Pullback of the Fubini-Study form  $\omega$ }

Let $\nu$ be an ergodic dilating measure of exponents $\lambda_1 > \lambda_2$. Let $(Z,W)_\epsilon$ be Oseledec-Poincar\'e coordinates for $(f,\nu)$. Let $n_4(\beta)$ be the smallest integer $n$ such that $e^{- n \epsilon} \leq \beta^{-1}$.

\begin{propo}
\label{lemmepullback}  Let $\hat{x} \in \hat{F}$ such that $\eta_{\epsilon}(\hat{x}) \geq \eta$ and $\beta_{\epsilon}(\hat{x}) \leq \beta$. If $n \geq \max \{n_4(\beta) , n_\epsilon(\hat x_n) \}$ and if $r \leq \eta$, then we have on $f^{-n}_{\hat{x}_n}(B_{x_n}(r e^{- n \epsilon}))$:  
\begin{enumerate}
\item $ (f^n)^* \omega \geq e^{- 4 n \epsilon + 2 n \lambda_1 } \left( \frac{i}{2} dZ_{\hat{x}}^{\epsilon} \wedge d \bar{Z_{\hat{x}}^{\epsilon}} \right)$ if the exponents do not resonate.
\item $ (f^n)^* \omega \geq e^{- 4 n \epsilon + 2 n \lambda_2 } \left( \frac{i}{2} dW_{\hat{x}}^{\epsilon} \wedge d \bar{W_{\hat{x}}^{\epsilon}} \right)$.
\end{enumerate}
\end{propo}

The remainder of this Section is devoted to the proof. Theorem \ref{FN1} gives  
$$(f^n)^* \omega = (\xi_{\hat{x}}^{\epsilon})^* ((R_{n,\hat{x_n}})^{-1})^* ((\xi_{\hat{x}_n})^{-1})^* \omega  $$
on $f^{-n}_{\hat{x}_n}(B_{x_n}(r e^{- n \epsilon}))$. Let $\omega_0 := \frac{i}{2} dz \wedge d \bar{z} +  \frac{i}{2} dw \wedge d \bar{w}$ be the standard form on $\bbd^2$. 

\begin{lemme} \label{pullbackl1}  Let $\hat{x} \in \hat{F}$ such that $\eta_{\epsilon}(\hat{x}) \geq \eta$ and $\beta_{\epsilon}(\hat{x}) \leq \beta$. For every $n \geq n_4(\beta)$ and $r \leq \eta$, we have on $\xi_{\hat{x}_n}^\epsilon (B_{x_n}(r e^{- n \epsilon}))$:
$$ e^{  -2 n \epsilon} \omega_0   \leq ((\xi_{\hat{x}_n}^\epsilon)^{-1})^* \omega \leq 2\omega_0 . $$
\end{lemme}

\begin{proof}
For every $p=(z,w)$ and $p'=(z',w')$ in $\xi_{\hat{x}_n}^\epsilon (B_{x_n}(r e^{- n \epsilon}))$, we have
\[ e^{-n \epsilon}  \beta^{-1} d(p,p') \leq \abs{(\xi_{\hat{x}_n}^{\epsilon})^{-1}(p) - (\xi_{\hat{x}_n}^{\epsilon})^{-1}(p')} \leq 2 d(p,p') .\]
This implies for every $n\geq n_4(\beta)$ and $ (z,w) \in \xi_{\hat{x}_n}^\epsilon (B_{x_n}(r e^{- n \epsilon}))$:
$$ \forall u \in \bbc^2, \quad e^{-2n \epsilon}  \abs{u}    \leq \abs{D_{(z,w)} (\xi_{\hat{x}_n}^{\epsilon})^{-1} (u)} \leq 2 \abs{u} .$$
This provides the desired estimates.
\end{proof}
 
\begin{lemme}\label{pullbackl2} Let $\hat{x} \in \hat{F}$. If $n \geq n_\epsilon(\hat{x}_n)$, then
\begin{enumerate}
\item $((R_{n,\hat{x}_n})^{-1})^* \omega_0 \geq e^{2(n \lambda_1 - n \epsilon)} \frac{i}{2} dz \wedge d \bar{z}$ if the exponents do not resonate.
\item $((R_{n,\hat{x}_n})^{-1})^* \omega_0 \geq e^{2(n \lambda_2 - n \epsilon)} \frac{i}{2} dw \wedge d \bar{w}$.
\end{enumerate}
\end{lemme}

\begin{proof}
We use the fact that the linear part of $R_{n,\hat{x}_n}$ is diagonal with coefficients $e^{ - n \epsilon - n \lambda_1} \leq \abs{\alpha_{n,\hat{x}_n}} \leq e^{ n \epsilon - n \lambda_1}$ and $ e^{ - n \epsilon - n \lambda_2} \leq \abs{\beta_{n,\hat{x}_n}} \leq e^{ n \epsilon - n \lambda_2}$ (see Theorem \ref{FN1}) and the fact that the $(1,1)$-forms $\frac{i}{2} dz \wedge d \bar{z}$ and $ \frac{i}{2} dw \wedge d \bar{w}$ are positive.
\end{proof}

To end the proof of Proposition \ref{lemmepullback}, we observe that for every  $\hat{x} \in \hat{F}$:
$$ 
(\xi_{\hat{x}}^{\epsilon})^* ( \frac{i}{2} dz \wedge d \bar{z}) = ( \frac{i}{2} dZ_{\hat{x}}^{\epsilon} \wedge d \bar{Z_{\hat{x}}^{\epsilon}} ) \qquad
(\xi_{\hat{x}}^{\epsilon})^* ( \frac{i}{2} dw \wedge d \bar{w}) = ( \frac{i}{2} dW_{\hat{x}}^{\epsilon} \wedge d \bar{W_{\hat{x}}^{\epsilon}} )$$
which follows from the definitions of $Z_{\hat{x}}^{\epsilon}$ and $W_{\hat{x}}^{\epsilon}$.

\subsection{Uniformizations} \label{sectunif}

Let $\nu$ be an ergodic dilating measure of exponents $\lambda_1 > \lambda_2$ and let $\delta > 0$. Let $\epsilon > 0$ and let $(Z,W)_\epsilon$ be Oseledec-Poincar\'e coordinates for $(f,\nu)$. 

\paragraph{Measure of dynamical balls}~~

We apply Lemma \ref{BK}. There exist $r_0 >0$, $n_2 \geq 1$ and $C \subset \bbp^2$ such that $\nu(C) \geq 1-\delta/8$ and for every $x \in C$ and $n \geq n_2$:
 \[ \nu (B_n(x,r_0/8)) \geq e^{-n h_{\nu} - \epsilon n}  ,  \]
\[\forall r \leq r_0 \ , \  \nu (B_n(x,5r)) \leq \nu (B_n(x,5 r_0)) \leq e^{-n h_{\nu} + \epsilon n} . \]
We denote $\Lambda^{(1)}  := \pi_0^{-1}(C) \cap \hat F$.

\paragraph{Control of the functions $n_\epsilon, \rho_{\epsilon} , L_{\epsilon} , \eta_{\epsilon} , \beta_{\epsilon}$ of Theorem \ref{FN1}}~~

Let $n_0$, $\rho_0 > 0$, $L_0 >0$, $\eta_0 > 0$ and $\beta_0 >0$ such that
\begin{equation}
\label{controlerayons}
\Lambda^{(2)} :=  \set{\hat{x} \in \hat{F},~n_\epsilon(\hat{x}) \leq n_0, ~\rho_{\epsilon}(\hat{x}) \geq \rho_0,~L_{\epsilon}(\hat{x}) \leq L_0,~\eta_{\epsilon}(\hat{x}) \geq \eta_0,~\beta_{\epsilon}(\hat{x}) \leq \beta_0}
\end{equation}
satisfies $\hat{\nu}(\Lambda^{(2)}) \geq 1 - \delta/8$.

\paragraph{Uniformization of the dimension of the current.}~~
 
Let $S$ be a positive closed current on $\bbp^2$ whose support contains the support of $\nu$. Let $r_1 > 0$ such that 
 $$\Lambda^{(3)} :=  \set{\hat{x} \in \hat{F} , ~~ \forall r \leq r_1, ~ \begin{array}{c}
r^{\bar{d_{S,Z}}(\hat x) + \epsilon}  \leq (S \wedge (\frac{i}{2} d Z_{\hat{x}}^{\epsilon} \wedge d \bar{Z_{\hat{x}}^{\epsilon}}))(B_ {x_0}(r)) \leq r^{\underline{d_{S,Z}}(\hat x) - \epsilon}  \\
(S \wedge (\frac{i}{2} d W_{\hat{x}}^{\epsilon} \wedge d \bar{W_{\hat{x}}^{\epsilon}}))(B_{x_0}(r)) \leq r^{\underline{d_{S,W}}(\hat x) - \epsilon} 
\end{array} }
  $$
satisfies $\hat{\nu}(\Lambda^{(3)}) \geq 1 - \delta/8$. In the case of the Green current $T$, the functions $\bar{d_{T,Z}}$, $\underline{d_{T,Z}}$ and $\underline{d_{T,W}}$ are $\hat \nu$-almost everywhere constant and denoted $\bar{d_{T,Z}}(\nu)$, $\underline{d_{T,Z}}(\nu)$ and $\underline{d_{T,W}}(\nu)$.

\paragraph{Uniformization of the dimension of the measure}~~

The lower dimension $\underline{d_{\nu}}$ is defined in Section \ref{BDM}.
Let $r_2 > 0$ such that
$$  D  := \lbrace x \in \bbp^2 \, , \, \forall r \leq r_2,~ \nu(B_x(r)) \leq r^{\underline{d_{\nu}} - \epsilon} \rbrace$$
satisfies $\nu(D) \geq 1 - \delta/8$. We set $\Lambda^{(4)} :=  \pi_0^{-1}(D) \cap \hat F$.

\paragraph{Definition of $\hat{\Lambda}_{\epsilon}, \eta_1$ and $N_\epsilon$.}~~

The integers $n_1(L), n_3(L)$ and $n_4(\beta)$ were defined before Lemma \ref{lemmeinclusion2}, \ref{lemmeinclusion1} and Proposition~\ref{lemmepullback}. Let $n_5$ be the smallest integer $n$ such that $e^{- n \epsilon} \leq 1/2$ and $2 e^{- n ( \lambda_1 + \epsilon ) } < 1$.  
$$\hat{\Lambda}_{\epsilon} := \Lambda^{(1)} \cap \Lambda^{(2)}  \cap \Lambda^{(3)} \cap \Lambda^{(4)}  , $$
 $$  \eta_1 := \min \{ \eta_0, r_0,  r_1, r_2 \} , $$
$$ N_\epsilon := \max \{ n_0, n_1(L_0), n_2 , n_3(L_0) ,n_4(\beta_0),n_5 \} .$$
We have $\hat{\nu}(\hat{\Lambda}_{\epsilon}) \geq 1 - \delta / 2$.

\paragraph{Definition of $\hat \Delta^n_\epsilon$.}~~

We set $$\forall n \geq N_\epsilon \ \ , \ \ \hat \Delta^n_\epsilon := \hat F \cap \hat f ^{-n} \{ n_\epsilon(\hat x) \leq n \}  = \{ \hat x \in \hat F  \ , \  n_\epsilon(\hat{x}_n ) \leq n  \}  . $$
Since $\hat \nu$ is $\hat f$-invariant and $\Lambda^{(3)} \subset \{ n_\epsilon(\hat x) \leq n \}$, we have $\hat \nu(\hat \Delta^n_\epsilon) \geq \hat \nu (\Lambda^{(3)}) \geq 1 - \delta / 8$. Hence
$$\forall n \geq N_\epsilon \ \ , \ \  \hat \nu (\hat{\Lambda}_{\epsilon} \cap \hat \Delta^n_\epsilon) \geq 1-\delta. $$

\section{Separated sets}\label{ES}

A subset $\{ x_1, \ldots, x_N \} \subset \bbp^2$ is $r$-separated if $d(x_i,x_j) \geq r$ for every $i \neq j$.
For $A \subset \bbp^2$, a subset $\set{x_1, \dots, x_N} \subset A$ is maximal $r$-separated with respect to  $A$ if it is $r$-separated and if for every $y \in A$, there exists $i \in \{ 1, \ldots N\}$ such that $d(y,x_i) < r$.
We use similar definitions for the distance $d_n$, in which case we say that the subsets are $(n,r)$ separated.

\subsection{Elementary separation}

\begin{lemme} \label{pointss\'epar\'es1} Let $f$ be an endomorphism of  $\bbp^2$ of degree $d \geq 2$ and let $\nu$ be an ergodic measure.
Let $A \subset \pi_0( \hat{\Lambda}_{\epsilon})$ such that $\nu(A) >0$ and let $c \in \left] 0 ,1 \right]$. Let $n \geq N_\epsilon$ and let $\set{x_1, \dots x_{N_n}} \subset A$ be maximal $(n , c \, \eta_1)$-separated with respect to $A$. Then
\begin{enumerate}
\item for every $i \neq j$, $B_n (x_i, c \, \eta_1 / 2) \cap B_n (x_j, c \, \eta_1 / 2) = \emptyset$.
\item $A \subset \cup_{i =1}^{N_n} B_n(x_i, c \, \eta_1)$.
\item $\nu(B_n(x_i, c \, \eta_1)) \leq e^{-n h_{\nu} + n \epsilon}$.
\item $e^{-n h_{\nu} - n \epsilon} \leq \nu(B_n(x_i, c \, \eta_1))$ si $c \geq 1/8$.
\item $N_n \geq \nu(A) e^{n h_{\nu} - n \epsilon}$.
\end{enumerate}
\end{lemme}

\begin{proof}
Item 1 comes from separation, Item 2 from the maximal property, Items 3 and 4 from Section \ref{sectunif}, because $n \geq N_\epsilon$, $c \, \eta_1 \leq \eta_1$ and $x_i \in C$. Items 2 and 3 then imply  $\nu(A) \leq \sum_{i=1}^{N_n} \nu (B_n(x_i, c \, \eta_1)) \leq N_n e^{-n h_{\nu} + n \epsilon}$, which gives Item 5.
\end{proof}

\subsection{Concentrated separation} \label{SCPI}
 
 Lemma \ref{pointss\'epar\'es1} applied with $c = 1/4$ gives $\nu(B_n(x_i, \eta_1 / 4)) \geq e^{-n h_{\nu} - n \epsilon}$ for every $x_i$ in a maximal $(n, \eta_1/4)$-separated subset of $A$. We shall see that it is possible to select a large number of $x_i$ such that
$$\nu(B_n(x_i, \eta_1 / 4) \cap A ) \geq e^{-n h_{\nu} - 2 n \epsilon}.$$
We take the arguments of de Th\'elin-Vigny in \cite[Section 6]{detvig15}.
Let $n_\delta$ be the smallest integer $n$ such that $e^{-n \epsilon} \leq \delta/2$.

\begin{lemme} \label{pointss\'epar\'es2} Let $A \subset \pi_0( \hat{\Lambda}_{\epsilon})$ such that $\nu(A) \geq  \delta$. For every $n \geq \max \{ N_\epsilon, n_\delta \}$, there exists a $(n,\eta_1/4)$-separated subset $\{ x_1, \dots, x_{N_{n,2}} \}$ of $A$ such that

\begin{enumerate}
\item for every $i \neq j$, $B_n (x_i, \eta_1 / 8) \cap B_n (x_j, \eta_1 / 8) = \emptyset$.
\item for every $1 \leq i \leq N_{n,2}$, $\nu(B_n(x_i,\eta_1/4) \cap A) \geq e^{-n h_{\nu} - 2  \epsilon n}$.
\item  $N_{n,2} \geq \nu(A) e^{n h_{\nu} - 2 n \epsilon}$.
\item $e^{-n h_{\nu} - n \epsilon} \leq \nu (B_n(x_i, c \eta_1))$ si $c \geq 1/8$.
\end{enumerate}
\end{lemme}

\begin{proof}
Let us apply  Lemma \ref{pointss\'epar\'es1} with $c=1/4$ and $n \geq \max \{ N_\epsilon, n_\delta \}$. There exists a  maximal $(n,\eta_1/4)$-separated subset $\{ x_1, \dots, x_{N_{n,1}} \}$ of $A$ satisfying:

- for every $i \neq j$, $B_n (x_i, \eta_1 / 8) \cap B_n (x_j, \eta_1 / 8) = \emptyset$,

- $A \subset \cup_{i =1}^{N_{n,1}} B_n(x_i, \eta_1/4)$,

- $e^{-n h_{\nu} - n \epsilon} \leq \nu (B_n(x_i,  \eta_1 / 8 ))$,

- $N_{n,1} \geq \nu(A) e^{n h_{\nu} - n \epsilon}$ .

Let us set $I := \set{1 \leq i \leq N_{n,1} \ , \  \nu(B_n(x_i,\eta_1/4) \cap A) \geq e^{-n h_{\nu} - 2  \epsilon n} }$. Let $N_{n,2}$ be the cardinality of $I$,  and assume that $I = \ensent{1,N_{n,2}}$ (we may adapt the sums below if $N_{n,2} = 0$). We want to bound $N_{n,2}$ from below. We know that $A \subset \cup_{i =1}^{N_{n,1}} B_n(x_i, \eta_1/4)$, hence
$$ \nu(A) \leq \sum_{i =1}^{N_{n,2}} \nu(  B_n(x_i, \eta_1/4) \cap A ) + \sum_{i = N_{n,2}+1}^{N_{n,1}} \nu(  B_n(x_i, \eta_1/4) \cap A )  . $$
If $i \not\in \ensent{1,N_{n,2}}$, we have $\nu(  B_n(x_i, \eta_1/4) \cap A) < e^{-n h_{\nu} - 2 \epsilon n}$ by definition of $I$. Otherwise,  $\nu(  B_n(x_i, \eta_1/4)) \leq e^{-n h_{\nu} + \epsilon n} $ since $x_i \in C$. This implies 
\begin{equation} \label{calculN2}
\nu(A) \leq N_{n,2}  e^{-n h_{\nu} + \epsilon n} + (N_{n,1} - N_{n,2}) e^{-n h_{\nu} - 2 \epsilon n} . 
\end{equation}
Let us give an upper bound for $N_{n,1}$. Since the balls $B_n (x_i, \eta_1 / 8)$ are pairwise disjoint and since   $\nu( B_n (x_i, \eta_1 / 8)) \geq e^{-n h_{\nu} - \epsilon n}$, we get
$$ e^{n h_{\nu} + \epsilon n} \geq N_{n,1} \geq N_{n,1} - N_{n,2}. $$
Combining this and (\ref{calculN2}),  we obtain
$$\nu(A) \leq N_{n,2}  e^{-n h_{\nu} + \epsilon n} + e^{-\epsilon n}.$$ 
Since $n \geq  n_\delta$, we have $e^{ -  n \epsilon }  \leq \delta / 2 \leq \nu(A)/2$, and hence
$N_{n,2} \geq \nu(A) e^{n h_{\nu} - \epsilon n}  / 2$. Finally $N_{n,2} \geq \nu(A) e^{n h_{\nu} - 2 \epsilon n}$ since $n \geq N_\epsilon \geq n_5$.  
\end{proof}

Now we  put in $B_n(x,\eta_1 / 2)$ a lot of balls whose centers are in $B_n(x, \eta_1 /4) \cap A$.

\begin{lemme} \label{pointss\'epar\'es3}
Let $A \subset \pi_0( \hat{\Lambda}_{\epsilon})$ such that $\nu(A) > 0$. Let $x \in A$ and let $n \geq N_\epsilon$ such that
$$\nu(B_n(x, \eta_1 /4) \cap A) \geq e^{ -n  h_{\nu} - 2 n \epsilon}.$$ 
Let $\set{y_1, \dots, y_{M_n}}$ be a maximal $2 \eta_1 e^{- n \lambda_1 - 4 n \epsilon}$-separated subset in $B_n(x, \eta_1 /4) \cap A$. 
\begin{enumerate}
\item for every $i \neq j$, $B(y_i,\eta_1  e^{-n \lambda_1 - 4 n \epsilon}) \cap B(y_j, \eta_1 e^{-n \lambda_1 - 4 n \epsilon}) = \emptyset$.
\item $B_n(x, \eta_1 /4) \cap A \subset \cup_{i=1}^{M_n} B(y_i,2 \eta_1 e^{-n \lambda_1 - 4 n \epsilon})$.
\item $B(y_i, \eta_1 e^{-n \lambda_1 - 4 n \epsilon}) \subset B_n(x,\eta_1 / 2)$.
\item $M_n \geq e^{ -n h_{\nu} - 2 n \epsilon} \left( \frac{1}{2 \eta_1} e^{n \lambda_1 + 4 n \epsilon} \right)^{\underline{d_{\nu}} - \epsilon}$.
\end{enumerate}
\end{lemme}

\begin{proof}
Item 1 comes from separation, Item 2 from the maximal property.
Lemmas \ref{lemmeinclusion2} then \ref{lemmeinclusion1} give
$$B(y_i, \eta_1 e^{-n \lambda_1 - 4 n \epsilon})\subset   f^{-n}_{\hat{y}_{i,n}}(B_{y_{i,n}}( \eta_1 e^{- 2 n \epsilon} /4 )    \subset B_n(y_i,\eta_1 / 4) .$$ 
Since $y_i \in B_n(x, \eta_1 /4)$, we have $B_n(y_i,\eta_1 / 4) \subset B_n(x,\eta_1 / 2)$, which yields Item 3. Item 2 implies
$$ \nu(B_n(x, \eta_1 /4) \cap A) \leq \sum_{i=1}^{M_n} \nu(B(y_i, 2 \eta_1 e^{-n \lambda_1 - 4 n \epsilon})) . $$
By assumption, the left hand side is larger than $e^{ -n  h_{\nu} - 2 n \epsilon}$. For the right hand side, since $n \geq N_\epsilon \geq n_5$, we have $ 2 \eta_1 e^{-n \lambda_1 - n \epsilon}  < \eta_1 \leq  r_2$ and thus
$$\nu(B(y_i, 2 \eta_1 e^{-n \lambda_1 - 4 n \epsilon})) \leq (2  \eta_1 e^{-n \lambda_1 - 4 n \epsilon})^{\underline{d_{\nu}} - \epsilon} $$
by using $y_i \in A \subset \pi_0( \hat{\Lambda}_{\epsilon}) \subset D$. This shows  $e^{ -n  h_{\nu} - 2 n \epsilon} \leq M_n (2 \eta_1 e^{-n \lambda_1 - 4 n \epsilon})^{\underline{d_{\nu}} - \epsilon}$.
\end{proof}

\section{Lower bounds for the directional dimensions of $T$}\label{MIDD}

Let  $\nu$ be an ergodic dilating measure whose exponents $\lambda_1 > \lambda_2$ do not resonate. Let $\epsilon > 0$ and let $(Z,W)_\epsilon$ be Oseledec-Poincar\'e coordinates for $(f,\nu)$.  We have $\bar{d_{T,Z}}(\nu) := \bar{d_{T,Z}}(\hat{x})$ and $\bar{d_{T,W}}(\nu) = \bar{d_{T,W}}(\hat{x})$ for $\hat \nu$-almost every $\hat x$ (see Proposition~\ref{propinvdimT2}). In this Section we prove Theorems \ref{prems} and \ref{resultat2}.
These results specify in a directional way Theorems A and B of de Th\'elin-Vigny \cite{detvig15} concerning the dimension of  $T$. We use below arguments of \cite{detvig15} by replacing the lower bound (obtained in  \cite{detvig15} by slicing  arguments)
$$(f^n)^* \omega \geq e^{2 n \lambda_2} e^{-n \epsilon} \omega$$
by the two lower bounds (obtained by normal forms Theorem \ref{FN1})
$$(f^n)^* \omega \geq e^{2 n \lambda_1} e^{- n \epsilon} d Z \wedge d \bar{Z}~\text{and}~(f^n)^* \omega \geq e^{2 n \lambda_2}  e^{- n \epsilon} d W \wedge d \bar{W}.$$ 
Theorem \ref{prems} uses elementary separation   (Lemma \ref{pointss\'epar\'es1}),  Theorem \ref{resultat2} uses concentrated separation (Lemmas \ref{pointss\'epar\'es2} and  \ref{pointss\'epar\'es3}). 
Theorem \ref{resultat2}  implies Theorem \ref{resultat2intro} (via the lower bound (\ref{dupdup}) for $\underline{d_{\nu}}$) and Theorem \ref{thmAintro}. 

\subsection{First lower bounds for the upper dimensions $d_{T,Z}$ et $d_{T,W}$}
\label{section minoration}

\begin{thm}\label{prems} Let $f$ be an endomorphism of $\bbp^2$ of degree $d \geq 2$. Let $\nu$ be an ergodic dilating measure whose exponents $\lambda_1 > \lambda_2$ do not resonate. There exist functions $O_5(\epsilon), O_6(\epsilon)$ satisfying the following properties. Let  $\epsilon > 0$ and let $(Z,W)_\epsilon$ be Oseledec-Poincar\'e coordinates for $(f,\nu)$. Then
\label{resultat} 
$$\bar{d_{T,Z}}(\nu) \geq 2 + \frac{h_{\nu} - \log d}{\lambda_1} - O_5(\epsilon) $$
$$\bar{d_{T,W}}(\nu) \geq 2 \frac{\lambda_2}{\lambda_1} + \frac{h_{\nu} - \log d}{\lambda_1} - O_6(\epsilon) .$$
\end{thm}

\begin{proof}
Let us denote $\bar{d_{T,Z}} : = \bar{d_{T,Z}}(\nu)$. For first estimate, we are going to show 
\begin{equation} \label{eq018}
(\lambda_1 + 4 \epsilon) (\bar{d_{T,Z}} - 2 + \epsilon) + 13 \epsilon  \geq h_{\nu} - \log d 
\end{equation}
which provides
$$ \bar{d_{T,Z}}  \geq 2 + {h_\nu - \log d - 13 \epsilon \over \lambda_1 + 4 \epsilon} - \epsilon  =:  2 + \frac{h_{\nu} - \log d}{\lambda_1} - O_5(\epsilon)  .$$ 
Let $\delta > 0$. Let $\hat \Lambda_\epsilon$ and $N_\epsilon$ be given by Section \ref{sectunif}. For every $n \geq N_\epsilon$ we set $A_n := \pi_0( \hat{\Lambda}_{\epsilon} \cap \hat \Delta^n_\epsilon)$, it satisfies $\nu(A_n) \geq \hat \nu(\hat \Lambda_\epsilon \cap \hat \Delta^n_\epsilon) \geq 1-\delta > 0$.  Lemma \ref{pointss\'epar\'es1} applied with $c=1$ yields a maximal $(n,\eta_1)$-separated subset $\{ x_1, \dots, x_{N_n} \}$ of $A_n$ with
 \begin{equation} \label{estimationNn}
 N_n  \geq \nu(A_n) e^{n h_{\nu} - n \epsilon} \geq (1-\delta) e^{n h_\nu - n \epsilon}. 
 \end{equation}
For every $i$, let us choose $\hat{x}_i \in \hat{\Lambda}_{\epsilon} \cap \hat \Delta^n_\epsilon$ such that $\pi_0(\hat{x}_i) = x_i$. From Proposition \ref{cohomologiecourants}, we get $d^n = \int_{\bbp^2} ( (f^n)_* T ) \wedge \omega$. Therefore
$$ d^n \geq \sum_{i=1}^{N_n} \int_{\bbp^2} (f^n)_* (1_{B_n(x_i,\eta_1 / 2)} T) \wedge \omega  \geq \sum_{i=1}^{N_n} \int_{\bbp^2} (1_{B_n(x_i,\eta_1 / 2)} T) \wedge (f^n)^* \omega . $$
By using Lemmas \ref{lemmeinclusion2} and \ref{lemmeinclusion1} with $\hat{x}_i \in \hat{\Lambda}_{\epsilon}$, we get
\begin{equation}\label{trezor}
B_{x_i}(\frac{\eta_1}{2} e^{-n \lambda_1 - 4 n \epsilon}) \subset   f^{-n}_{\hat{x}_{i,n}}(B_{x_{i,n}}(\frac{\eta_1}{8} e^{- 2 n \epsilon}))    \subset  B_n(x_i, {\eta_1 \over 8})   .
\end{equation}
Since $T \wedge (f^n)^* \omega$ is a positive measure, we deduce 
\[ d^n \geq \sum_{i=1}^{N_n} \int_{\bbp^2} (1_{B_{x_i}(\frac{\eta_1}{2}e^{-n \lambda_1 - 4 n \epsilon})} T) \wedge (f^n)^* \omega. \]
Thanks to the first inclusion of (\ref{trezor}) and $\hat x_i \in \hat \Delta^n_\epsilon$ (which implies $n \geq n_\epsilon(\hat x_{i,n})$), we can use  Proposition \ref{lemmepullback} to bound $(f^n)^* \omega$ from below. We obtain 
$$ d^n \geq \sum_{i = 1}^{N_n} e^{2 n \lambda_1 - 4 n \epsilon} \left( T \wedge \frac{i}{2} dZ_{\hat{x}_i}^{\epsilon} \wedge d \bar{Z_{\hat{x}_i}^{\epsilon}} \right)(B_{x_i}( \frac{\eta_1}{2} e^{ - n \lambda_1 - 4 n \epsilon})). $$
Since  $\hat{x_i} \in \hat{\Lambda}_{\epsilon} \subset \Lambda^{(3)}$ and $\eta_1 \leq r_1$, we deduce
$$d^n \geq \sum_{i=1}^{N_n} e^{ 2 n \lambda_1 - 4 n \epsilon} \left(\frac{\eta_1}{2}e^{-n \lambda_1 - 4 n \epsilon}\right)^{\bar{d_{T,Z}} + \epsilon}.$$
Finally, we use the estimate (\ref{estimationNn}):
$$ d^n \geq \nu(A_n)  \left(\frac{\eta_1}{2}\right)^{\bar{d_{T,Z}^{\epsilon}} + \epsilon} e^{n ( h_{\nu} -  ( \lambda_1 + 4  \epsilon)(\bar{d_{T,Z}^{\epsilon}} - 2 + \epsilon) - 13 \epsilon )} , $$
where $\nu(A_n) \geq (1 - \delta)$. By taking logarithm and dividing by $n$, we get (\ref{eq018}) when $n \to \infty$. The second estimate concerning the coordinate $W$ is proved in a similar way, by using Proposition~\ref{lemmepullback} to bound $(f^n)^* \omega$ from below. We precisely get
\begin{equation}
(\lambda_1 + 4 \epsilon) (\bar{d_{T,W}}  + \epsilon) + 5 \epsilon  \geq h_{\nu} - \log d + 2 \lambda_2 
\end{equation}
which yields
$$ \bar{d_{T,W}}  \geq {2 \lambda_2 \over \lambda_1 + 4 \epsilon} + {h_\nu - \log d  -5 \epsilon  \over \lambda_1 + 4 \epsilon} - \epsilon  =:  2{\lambda_2 \over \lambda_1} + \frac{h_{\nu} - \log d}{\lambda_1} - O_6(\epsilon)  .$$ 
This completes the proof of Theorem \ref{prems}.
\end{proof}

\subsection{Proof of Theorem \ref{resultat2intro}} \label{section minoration fine}

Theorem \ref{resultat2intro} is a consequence of Theorem \ref{resultat2} below and of the lower bound (\ref{dupdup}). 

\begin{thm}\label{resultat2} 
Let $f$ be an endomorphism of $\bbp^2$ of degree $d \geq 2$. Let $\nu$ be an ergodic dilating measure whose exponents $\lambda_1 > \lambda_2$ do not resonate. There exist functions $O_1(\epsilon), O_2(\epsilon)$ satisfying the following properties. Let  $\epsilon > 0$ and let $(Z,W)_\epsilon$ be Oseledec-Poincar\'e coordinates for $(f,\nu)$. Then
$$ \bar{d_{T,Z}}(\nu) \geq 2 + \underline{d_{\nu}} - \frac{\log d}{\lambda_1} - O_1(\epsilon),$$
$$ \bar{d_{T,W}}(\nu) \geq 2 \frac{\lambda_2}{\lambda_1} + \underline{d_{\nu}} - \frac{\log d}{\lambda_1} - O_2(\epsilon).$$
\end{thm}

\begin{proof} Let us set $\bar{d_{T,Z}} := \bar{d_{T,Z}}(\nu)$. We are going to show  
\begin{equation} \label{eq019}
(\lambda_1 + 4 \epsilon)(\bar{d_{T,Z}}- \underline{d_{\nu}} + 2 \epsilon) + 8 \epsilon \geq 2 \lambda_1 - \log d.
\end{equation}
This yields as desired 
\begin{equation} \label{ooee}
\bar{d_{T,Z}}- \underline{d_{\nu}} \geq {2 \lambda_1 \over \lambda_1 + 4 \epsilon} - {\log d - 8 \epsilon \over \lambda_1 + 4 \epsilon} - 2 \epsilon =: 2 - {\log d \over \lambda_1} - O_1(\epsilon)  .
\end{equation}
Let $\delta > 0$. Let $\hat \Lambda_\epsilon$ and $N_\epsilon$ be given by Section \ref{sectunif}. For every $n \geq N_\epsilon$ we set $A_n := \pi_0(\hat \Lambda_\epsilon \cap \hat \Delta^n_\epsilon)$, it satisfies $\nu(A_n) \geq 1-\delta$. Let $\{ x_1 , \ldots , x_{N_{n,2}} \}$ be a $(n, \eta_1/4)$-separated subset of $A_n$ provided by Lemma \ref{pointss\'epar\'es2}. Then for every $x_i$, we set a $2 \eta_1 e^{-n \lambda_1 - 4 n \epsilon}$-separated subset $\{ y^i_{1}, \dots, y^i_{M_n} \}$ given by Lemma \ref{pointss\'epar\'es3}. For every $i$ we choose $\hat{x_i} \in \hat{\Lambda}_{\epsilon} \cap \hat \Delta^n_\epsilon$ such that $\pi_0(\hat{x}_i) = x_i$, and for every $j$ we choose $\hat{y}^i_j \in \hat{\Lambda}_{\epsilon} \cap \hat \Delta^n_\epsilon$ such that $\pi_0(\hat{y}^i_j) = y^i_{j}$. According to Proposition \ref{cohomologiecourants}, we have $d^n = \int_{\bbp^2}  (f^n)_* T  \wedge \omega$, thus
$$ d^n  \geq \sum_{i=1}^{N_{n,2}} \sum_{j=1}^{M_n} \int_{\bbp^2} (f^n)_* (1_{B(y^i_{j} , \eta_1 e^{-n \lambda_1 - 4 n \epsilon})} T) \wedge \omega  = \sum_{i=1}^{N_{n,2}} \sum_{j=1}^{M_n} \int_{B( y^i_{j} , \eta_1 e^{-n \lambda_1 - 4 n \epsilon} )} T \wedge (f^n)^* \omega .$$
Lemma \ref{lemmeinclusion2} with $\hat{y}^i_j \in \hat{\Lambda}_{\epsilon}$ implies
\begin{equation}\label{trez}
B(y^i_j , \eta_1 e^{-n \lambda_1 - 4 n \epsilon}) \subset   f^{-n}_{\hat{y}^i_{j,n}}(B( {y^i_{j,n}} , \frac{\eta_1}{4} e^{- 2 n \epsilon}))  .
\end{equation}
Since $\hat{y}^i_j \in \hat \Delta^n_\epsilon$, we can apply  Proposition \ref{lemmepullback} to bound $(f^n)^* \omega$ from below:
\begin{equation} \label{tronccommun}
d^n \geq \sum_{i=1}^{N_{n,2}} \sum_{j=1}^{M_n} e^{2n\lambda_1 - 4 n \epsilon} \left( T \wedge  \frac{i}{2} dZ_{\hat{y}^i_{j}}^{\epsilon} \wedge d \bar{Z_{\hat{y}^i_{j}}^{\epsilon}} \right) \left( B_{y^i_{j}}(\eta_1 e^{-n \lambda_1 - 4 n \epsilon}) \right).
\end{equation}
Now $\hat{y}^i_{j} \in \hat{\Lambda}_{\epsilon} \subset \Lambda^{(3)}$ and $n \geq N_\epsilon$, hence
$$d^n \geq \sum_{i=1}^{N_{n,2}} \sum_{j=1}^{M_n} e^{2n\lambda_1 - 4 n\epsilon} (\eta_1 e^{-n \lambda_1 - 4 n \epsilon})^{\bar{d_{T,Z}^{\epsilon}} + \epsilon} . $$
Finally, we use the lower bounds for $M_n$ (Lemma \ref{pointss\'epar\'es3}) and for $N_{n,2}$ (Lemma \ref{pointss\'epar\'es2}). We obtain for every $n \geq \max \{ N_\epsilon, n_{1-\delta} \}$:
\[ d^n \geq (1-\delta) e^{n h_{\nu} -  2 n \epsilon} \cdot e^{ -n h_{\nu} - 2 n \epsilon} \left( \frac{1}{2 \eta_1} e^{n \lambda_1 + 4 n \epsilon} \right)^{\underline{d_{\nu}} - \epsilon} \cdot e^{2n\lambda_1 - 4 n \epsilon} (\eta_1 e^{-n \lambda_1 - 4 n \epsilon})^{\bar{d_{T,Z}^{\epsilon}} + \epsilon} . \]
Let us note that the entropy $h_{\nu}$ disappear for the benefit of  $\underline{d_{\nu}}$, and we get
\begin{equation}
 d^n \geq c \,  e^{- 8 n \epsilon} \left( e^{n \lambda_1 + 4 n \epsilon} \right)^{\underline{d_{\nu}} - \bar{d_{T,Z}^{\epsilon}}- 2 \epsilon} e^{2n\lambda_1} ,
\end{equation}
where $c :=(1-\delta) \eta_1 ^{\bar{d_{T,Z}} + \epsilon}  / (2 \eta_1)^{\underline{d_{\nu}} - \epsilon}$. Taking logarithm and then dividing by $n$, we obtain (\ref{eq019}) when $n \to + \infty$. Similarly, we can prove 
 $$(\lambda_1 + 4 \epsilon)(\bar{d_{T,W}}- \underline{d_{\nu}} + 2 \epsilon) + 8 \epsilon \geq 2 \lambda_2 - \log d $$
by using again Proposition~\ref{lemmepullback} to bound $(f^n)^* \omega$ from below.
\end{proof}

\section{Currents $S$ and semi-extremal endomorphisms}\label{ECG}

\subsection{Proof of Theorem \ref{thm1intro}} \label{sectioncourantS}

Let $S$ be a $(1,1)$ closed positive current of $\bbp^2$. If $S$ does not satisfy $ f^* S = d S$, the directional dimensions may be not $\hat{\nu}$-almost everywhere constant (see Proposition \ref{propinvdimT2}). In this case, in the manner of de Th\'elin-Vigny  \cite{detvig15}, we take on an adapted  definition and obtain the following result. The functions $O_5(\epsilon), O_6(\epsilon)$ were defined in Theorem \ref{prems}.

\begin{thm}\label{resultat5} 
Let $f$ be an endomorphism of $\bbp^2$ of degree $d \geq 2$. Let $S$ be a $(1,1)$ closed positive current of $\bbp^2$, of mass $1$. Let $\nu$ be an ergodic dilating measure whose exponents $\lambda_1 > \lambda_2$ do not resonate. We assume that $\Sup(\nu) \subset \Sup S$. Let  $\epsilon > 0$ and let $(Z,W)_\epsilon$ be Oseledec-Poincar\'e coordinates for $(f,\nu)$. For every $\hat{\Lambda} \subset \hat{F}$ such that $\hat \nu(\hat \Lambda) > 0$, we set 
$$ \bar{d_{S,Z}}(\hat{\Lambda}) := \sup_{\hat{x} \in \hat{\Lambda}} \bar{d_{S,Z}} (\hat{x}), \qquad \bar{d_{S,W}}(\hat{\Lambda}) := \sup_{\hat{x} \in \hat{\Lambda}} \bar{d_{S,W}} (\hat{x}).$$
Then
$$\bar{d_{S,Z}}(\hat{\Lambda}) \geq 2 + \frac{h_{\nu} - \log d}{\lambda_1} - O_5(\epsilon),$$
$$\bar{d_{S,W}}(\hat{\Lambda}) \geq 2 \frac{\lambda_2}{\lambda_1} + \frac{h_{\nu} - \log d}{\lambda_1} -  O_6(\epsilon).$$
\end{thm}

\begin{proof}
Let $2\delta := \hat{\nu} (\hat{\Lambda})$. We construct  $\hat{\Lambda}_{\epsilon}$ et $N_\epsilon$ for the current $S$ as in Section \ref{sectunif}. We have $\hat \nu (\hat{\Lambda}_{\epsilon} \cap \hat \Delta^n_\epsilon) \geq 1 - \delta$ for every $n \geq N_\epsilon$, thus $\hat{\nu} (\hat{\Lambda} \cap \hat{\Lambda}_{\epsilon}  \cap \hat \Delta^n_\epsilon) \geq \delta > 0$. We follow the  arguments of Theorem \ref{resultat}. Lemma \ref{pointss\'epar\'es1} applied to $A_n : =  \pi_0 ( \hat{\Lambda} \cap \hat{\Lambda}_{\epsilon} \cap \hat \Delta^n_\epsilon)$ and $c=1$ provides a maximal  $(n,\eta_1)$-separated subset $\{ x_1, \dots, x_{N_n} \} $ of $A_n$. For every $i$, let us choose $\hat{x}_i \in \hat{\Lambda} \cap \hat{\Lambda}_{\epsilon} \cap \hat \Delta^n_\epsilon$ such that $\pi_0(\hat{x_i}) = x_i$. According to Proposition \ref{cohomologiecourants} we have $d^n = \int_{\bbp^2}  (f^n)_* S  \wedge \omega$, hence
$$ d^n = \int_{\bbp^2} ( (f^n)_* S ) \wedge \omega  \geq \sum_{i=1}^{N_n} \int_{\bbp^2} (1_{B_n(x_i,\eta_1 / 2)} S) \wedge (f^n)^* \omega .$$
Then we use the inclusions and the lower bound for  $(f^n)^* \omega$ to obtain: 
$$ d^n \geq \sum_{i = 1}^{N_n} e^{2 n \lambda_1 - 4 n \epsilon} \left( S \wedge \frac{i}{2} dZ_{\hat{x}_i}^{\epsilon} \wedge d \bar{Z_{\hat{x}_i}^{\epsilon}} \right)(B_{x_i}( \frac{\eta_1}{2} e^{ - n \lambda_1 - 4 n \epsilon})). $$
Since $\hat{x_i} \in \hat{\Lambda}_{\epsilon}$ and $n \geq N_\epsilon$, we get
$$d^n \geq \sum_{i=1}^{N_n} e^{ 2 n \lambda_1 - 4 n \epsilon} \left(\frac{\eta_1}{2}e^{-n \lambda_1 - 4 n \epsilon}\right)^{\bar{d_{S,Z}}(\hat x_i) + \epsilon}.$$
Now we use the adapted definition of $\bar{d_{S,Z}} (\hat{\Lambda})$ and the lower estimate (\ref{estimationNn}) to obtain
\begin{equation}
 d^n \geq \nu(A_n) \left(\frac{\eta_1}{2}\right)^{\bar{d_{S,Z}}(\hat{\Lambda}) + \epsilon} e^{n h_{\nu} - 13 n \epsilon - n( \lambda_1 + 4  \epsilon)(\bar{d_{S,Z}}(\hat{\Lambda}) - 2 + \epsilon) }  ,
\end{equation}
where $\nu(A_n) \geq \delta$. The lower bound concerning $W$ is proved in a similar way.
\end{proof}

\begin{rem} \label{remarque_resultat_5}
Theorem \ref{resultat5} is the counterpart of Theorem \ref{resultat} for currents $S$. Similarly, the counterpart of Theorem \ref{resultat2} can be proved with $n \geq \max\{N_\epsilon , n_\delta \}$ in the proof.
\end{rem}

\subsection{Proof of Theorem \ref{thmAintro}}\label{proofthmA}
 
Since $T$ is $f$-invariant and $\nu$ is ergodic,  we have $\bar{d_T}(x) = \bar{ d_T}(\mu)$ for $\mu$-almost every $x$, see Proposition \ref{ergodim}.
According to Proposition \ref{propdimmesuresabscon}, $\mu << \sigma_T$ implies
\begin{equation} \label{eq029}
\bar{d_T}(\mu) \leq \bar{d_{\mu}} .
\end{equation}
Let us analyse  these quantities.
On the one hand, Proposition~\ref{minmesures} yields $\bar{d_T}(\mu) = \min \set{\bar{d_{T,Z}}(x),\bar{d_{T,W}}(x)}$ for $\mu$-almost every $x \in \bbp^2$ and for every holomorphic coordinates $(Z,W)$ near $x$. On the other hand, since $\underline{d_{\mu}} = \bar{d_{\mu}}$, then $\underline{d_{\mu}} = \bar{d_{\mu}} = \dim_H(\mu)$, which is equal to $2 + \frac{\log d}{\lambda_1}$ by (\ref{corollairedimintro}). One deduces from (\ref{eq029}) that if $\mu << \sigma_T$, then
\begin{equation} \label{eq032}
\min \set{\bar{d_{T,Z}}(x),\bar{d_{T,W}}(x)} \leq 2 + \frac{\log d}{\lambda_1}.
\end{equation}
Now we use Theorem \ref{resultat2}. Let $\epsilon > 0$ such that $4 - O_1(\epsilon) > 2 + {\log d \over \lambda_1}$, where the function $O_1(\epsilon)$ is defined by (\ref{ooee}). Let $(Z,W)_\epsilon$ Oseledec-Poincar\'e coordinates for $(f,\mu)$. 

First we bound $\bar{d_{T,Z}}(\mu)$ from below, then we establish the formula for $\bar{d_{T,W}}(\mu)$ modulo the function $O_2(\epsilon)$. If $\mu << \sigma_T$, then $\lambda_2 = {1 \over 2}\log d$ by Theorem \ref{dujardin}. We deduce from  (\ref{dupdup}) that $2 + \frac{\log d}{\lambda_1} \leq \underline{d_{\mu}}$. Theorem \ref{resultat2} then provides  
$$\bar{d_{T,Z}}(\mu) \geq 4 - O_1(\epsilon) \ \ , \ \ \bar{d_{T,W}}(\mu)  \geq 2 + \frac{\log d}{\lambda_1} - O_2(\epsilon).$$
Finally, using $4 - O_1(\epsilon)> 2 + {\log d \over  \lambda_1}$ and (\ref{eq032}) with Oseledec-Poincar\'e  coordinates $(Z,W)_\epsilon$, we get $\bar{d_{T,W}}(\mu) \leq 2 + \frac{\log d}{\lambda_1}$ as desired.

\section{Upper bounds for the directional dimensions of $T$} \label{MDD}

In this Section we show Theorem \ref{dimcourantfinieintro}. Let $\nu$ be an ergodic dilating measure such that $\Sup(\nu) \subset \Sup (\mu)$ and whose exponents $\lambda_1 > \lambda_2$ do not resonate. Let $\epsilon > 0$ and let $(Z,W)_\epsilon$ be Oseledec-Poincar\'e coordinates for $(f,\nu)$. We want to prove
$$ \underline{d_{T,Z}}(\nu) \leq \frac{\log d}{\lambda_2} + 2 \frac{\lambda_1}{\lambda_2} + O_3(\epsilon) \ \ \textrm{ and }  \ \ \underline{d_{T,W}}(\nu) \leq \frac{\log d}{\lambda_2} + 2 + O_4(\epsilon). $$
We shall directly obtain these upper bounds for $\hat{\nu}$-almost every $\hat{x}$, by using the jacobians of $T \wedge d Z^{\epsilon}_{\hat{x}} \wedge d \bar{Z^{\epsilon}_{\hat{x}}}$ and $T \wedge d W^{\epsilon}_{\hat{x}} \wedge d \bar{W^{\epsilon}_{\hat{x}}}$ with respect to $f$. In particular, we shall not use separated subsets. The Monge-Amp\`ere equation $\mu = T \wedge T$ will be crucial. 

\subsection{Dimensions of the Green current on the equilibrium measure}\label{dimensionfinieducourant}

The following Proposition is proved in Section \ref{section Monge-Amp\`ere}.

\begin{propo} \label{mesuressubmersion}
Let $f$ be an endomorphism of $\bbp^2$ of degree $d \geq 2$ and let $T$ be its Green current. Let $x \in \Sup \mu$ and let $Z$ be a local holomorphic coordinate (submersion) in a neighbourhood $V$ of $x$. Then $T \wedge ( \frac{i}{2} d Z \wedge d \bar{Z})$ is not the zero measure on $V$.
\end{propo}

This implies: 

\begin{propo}\label{propcourantfini} 
Let $f$ be an endomorphism of $\bbp^2$ of degree $d \geq 2$. Let $\nu$ be an ergodic dilating measure of exponents $\lambda_1 > \lambda_2$ and whose support is contained in the support of $\mu$. Let  $\epsilon > 0$ and let $(Z,W)_\epsilon$ be Oseledec-Poincar\'e coordinates for $(f,\nu)$. We recall that for every $\hat x \in \hat F$, $(Z_{\hat{x}}^{\epsilon} , W_{\hat{x}}^{\epsilon})$ is defined on $B_{x_0} (\eta_\epsilon(\hat x))$. Then, for every $0 < r < \eta_{\epsilon}(\hat{x})$, 
$$\left[ T \wedge \frac{i}{2} d Z_{\hat{x}}^{\epsilon} \wedge d \bar{Z_{\hat{x}}^{\epsilon}} \right](B_x(r)) > 0
\ \ \textrm{ and }  \ \ \left[ T \wedge \frac{i}{2} d W_{\hat{x}}^{\epsilon} \wedge d \bar{W_{\hat{x}}^{\epsilon}} \right](B_x(r)) > 0.
$$
In particular, for every $\delta >0$, there exist $m_0 \geq 1$, $L_0 \geq 1$ and $q_0 \geq 1$ such that
$$ \hat{\Omega}_{\epsilon} := \set{ \eta_\epsilon (\hat{x}) \geq {1 \over 4m_0} \ , \  L(\hat x)\leq L_0 \ , \ \begin{array}{c} 
\left[ T \wedge \frac{i}{2} d Z_{\hat{x}}^{\epsilon} \wedge d \bar{Z_{\hat{x}}^{\epsilon}} \right](B_x(\frac{1}{4m_0})) \geq \frac{1}{q_0}  \\
\left[ T \wedge \frac{i}{2} d W_{\hat{x}}^{\epsilon} \wedge d \bar{W_{\hat{x}}^{\epsilon}} \right](B_x(\frac{1}{4m_0})) \geq \frac{1}{q_0} 
\end{array} }$$
satisfies $\hat{\nu}( \hat{\Omega}_{\epsilon}) \geq 1 - \delta$.
\end{propo}

\begin{proof}
The first part immediately follows from Proposition~\ref{mesuressubmersion}. To prove the second part, let $m_0 \geq 1$ and $L_0 \geq 1$ be such that $\hat{\nu} \set{\eta_{\epsilon} \geq \frac{1}{4m_0}} \cap \set {L \leq L_0 } \geq 1 - \delta / 2$. Then we choose $q_0$ large enough so that $\hat{\nu}( \hat{\Omega}_{\epsilon}) \geq 1 - \delta$.
\end{proof}

We define for every $n \geq 1$:
$$ \hat{\Omega}_{\epsilon}^n := \hat{\Omega}_{\epsilon} \cap \hat{f}^{-n}(\hat{\Omega}_{\epsilon}).$$
Since $\hat{\nu}$ is invariant, we have:
\begin{equation} \label{tailleomega}
\hat{\nu}(\hat{\Omega}_{\epsilon}^n) \geq 1 -  2 \delta.
\end{equation}

The following Proposition will be useful to prove Theorems \ref{dimcourantfinieintro} and \ref{resultat3}. $L_0$ is defined in Proposition \ref{propcourantfini} and  $n_1(L_0) \geq 1$ is defined before Lemma \ref{lemmeinclusion2}.

\begin{propo} \label{unifdimcourant}
Let $f$ be an endomorphism of $\bbp^2$ of degree $d \geq 2$.  Let $\nu$ be a ergodic dilating measure of  exponents $\lambda_1 >  \lambda_2$ and such that $\Sup(\nu) \subset \Sup ( \mu )$. For every $n \geq n_1(L_0)$ and $\hat{x} \in \hat{\Omega}_{\epsilon}^n$, we have:
$$ \left[ T \wedge \frac{i}{2} d Z_{\hat{x}}^{\epsilon} \wedge d \bar{Z_{\hat{x}}^{\epsilon}} \right]\left( B_x \left( \frac{1}{m_0}e^{-n \lambda_2 + 3 n \epsilon} \right) \right) \geq \frac{1}{d^n} e^{- 2 n \lambda_1 - 2 n \epsilon} \frac{1}{q_0} \quad \text{if } \lambda_1 \not\in \set{k \lambda_2, k \geq 2},$$ 
$$ \left[ T \wedge \frac{i}{2} d W_{\hat{x}}^{\epsilon} \wedge d \bar{W_{\hat{x}}^{\epsilon}} \right]\left( B_x \left( \frac{1}{m_0}e^{-n \lambda_2 + 3 n \epsilon} \right) \right) \geq \frac{1}{d^n} e^{- 2 n \lambda_2 - 2 n \epsilon} \frac{1}{q_0} \quad \text{ for every }\lambda_1 > \lambda_2.$$ 
\end{propo}

\begin{proof}  Let $ \hat{x} \in \hat{\Omega}^n_{\epsilon}$ and let
$$E_n := f_{\hat{x}_n}^{-n}(B_{x_n}(\frac{1}{4m_0})).$$
The inverse branch $f^{-n}_{\hat{x}_n}$ is well defined on $B_{x_n}(\frac{1}{4m_0})$ since $\hat{x}_n \in  \hat{\Omega}_{\epsilon} $. Let $g_n$ be the restriction of $f^n$ on $E_n$. By using $f^{-n}_{\hat{x}_n} \circ g_n = Id_{E_n}$ and $T = \frac{1}{d^n} g_n^*T$ on $E_n$, we obtain
\begin{align*}
T \wedge  \frac{i}{2} d Z_{\hat{x}}^{\epsilon} \wedge  d \bar{Z_{\hat{x}}^{\epsilon}} & = \frac{1}{d^n} g_n^*T \wedge g_n^* (f_{\hat{x}_n}^{-n})^* (\frac{i}{2} d Z_{\hat{x}}^{\epsilon} \wedge d \bar{Z_{\hat{x}}^{\epsilon}}) \\
& = \frac{1}{d^n} g_n^* \left[ T \wedge \frac{i}{2} (d Z_{\hat{x}}^{\epsilon}\circ (f_{\hat{x}_n}^{-n})) \wedge d( \bar{Z_{\hat{x}}^{\epsilon} \circ (f_{\hat{x}_n}^{-n})}) \right] 
\end{align*}
on the open subset $E_n$. Now we use (\ref{eq033}) to write $Z_{\hat{x}}^{\epsilon}\circ (f_{\hat{x}_n}^{-n}) = \alpha_{n,\hat{x}_n} Z_{\hat{x}_n}$. Since $\abs{\alpha_{n,\hat{x}_n}}^2 \geq e^{-2  n \lambda_1 - 2n \epsilon} $, we get on $E_n$:
\begin{equation} \label{calculdimcourant2}
T \wedge \frac{i}{2} d Z_{\hat{x}}^{\epsilon} \wedge d \bar{Z_{\hat{x}}^{\epsilon}} \geq  \frac{1}{d^n}  e^{- 2 n \lambda_1 - 2 n \epsilon} g_n^* \left[ T \wedge \frac{i}{2} d Z_{\hat{x}_n}^{\epsilon} \wedge d \bar{Z_{\hat{x}_n}^{\epsilon}} \right] .
\end{equation}
We are going to bound from above the left hand side and to bound from below the right hand side (applied to  $E_n$). Using Lemma \ref{lemmeinclusion2} with $r = 1/m_0$ and $n \geq L_0$ we obtain $E_n \subset B_x \left( \frac{1}{m_0}e^{-n \lambda_2 + 3 n \epsilon} \right)$, hence
\begin{equation} \label{calculdimcourant3}
T \wedge \frac{i}{2} d Z_{\hat{x}}^{\epsilon} \wedge d \bar{Z_{\hat{x}}^{\epsilon}} \left( B_x \left( \frac{1}{m_0}e^{-n \lambda_2 + 3 n \epsilon} \right) \right) \geq T \wedge \frac{i}{2} d Z_{\hat{x}}^{\epsilon} \wedge d \bar{Z_{\hat{x}}^{\epsilon}} (E_n)
\end{equation}
For the right hand side, since $g_n$ is injective on $E_n$ and $g_n(E_n) = B_{x_n} \left( \frac{1}{4 m_0} \right)$, we get
\begin{equation} \label{calculdimcourant4}
g_n^* \left[ T \wedge \frac{i}{2} (d Z_{\hat{x}_n}^{\epsilon} \wedge  \bar{d Z_{\hat{x}_n}^{\epsilon}}) \right] (E_n) = \left[ T \wedge \frac{i}{2} (d Z_{\hat{x}_n}^{\epsilon} \wedge ( \bar{d Z_{\hat{x}_n}^{\epsilon}}) \right] (B_{x_n}(\frac{1}{4 m_0})) \geq \frac{1}{q_0},
\end{equation}
where the inequality comes from $\hat{x}_n \in \hat{\Omega}_{\epsilon}$. By combining (\ref{calculdimcourant2}), (\ref{calculdimcourant3}) and (\ref{calculdimcourant4}) we obtain
$$ \left[ T \wedge \frac{i}{2} d Z_{\hat{x}}^{\epsilon} \wedge d \bar{Z_{\hat{x}}^{\epsilon}} \right]\left( B_x \left( \frac{1}{m_0}e^{-n \lambda_2 + 3 n \epsilon} \right) \right) \geq \frac{1}{d^n} e^{- 2 n \lambda_1 - 2  n \epsilon} \frac{1}{q_0}.$$ 
We use $W_{\hat{x}}^{\epsilon}\circ (f_{\hat{x}_n}^{-n}) = \beta_{n,\hat{x}} W_{\hat{x}_n}$ et  $\abs{\beta_{n,\hat{x}}}^2  \geq e^{- 2 n \lambda_2 - 2 n \epsilon}$ to prove the other lower bound.
\end{proof}

\subsection{Proof of Theorem \ref{dimcourantfinieintro}}

We take the notations of Section \ref{dimensionfinieducourant}. Let
$$\hat{\Omega}_{\epsilon} : = \limsup_{n \in \bbn}  \hat{\Omega}_{\epsilon} \cap \hat{f}^{-n}(\hat{\Omega}_{\epsilon})= \limsup_{n \in \bbn} \hat{\Omega}_{\epsilon}^n .$$
We have $\hat{\nu}(\hat{\Omega}_{\epsilon}) \geq 1 - 2 \delta$ according to (\ref{tailleomega}). Let $\hat{x} \in \hat{\Omega}_{\epsilon}$. Then there exists an increasing sequence of intergers $(l_p)_p$ such that  
$$\hat{x} \in  \hat{\Omega}_{\epsilon} \cap \hat{f}^{-l_p}(\hat{\Omega}_{\epsilon}) = \hat{\Omega}_{\epsilon}^{l_p}$$
for every $p \geq 0$. Proposition~\ref{unifdimcourant} then asserts for $p$ large enough:
$$ \left[ T \wedge \frac{i}{2} d Z_{\hat{x}}^{\epsilon} \wedge d \bar{Z_{\hat{x}}^{\epsilon}} \right]\left( B_x \left( \frac{1}{m_0}e^{-l_p \lambda_2 + 3 l_p \epsilon} \right) \right) \geq \frac{1}{d^{l_p}} e^{- 2 l_p \lambda_1 -  2 l_p \epsilon} \frac{1}{q_0}.$$ 
If $p$ is also large enough so that  $e^{l_p \epsilon} \geq {1 \over m_0}$ and $\frac{1}{q_0} \geq e^{- l_p \epsilon}$, we obtain with $r_p := e^{-l_p (\lambda_2 - 4 \epsilon)}$:
$$ \left[ T \wedge \frac{i}{2} d Z_{\hat{x}}^{\epsilon} \wedge d \bar{Z_{\hat{x}}^{\epsilon}} \right]\left( B_x \left(  r_p \right) \right) \geq e^{- l_p (\log d + 2 \lambda_1 + 3 \epsilon )} = r_p^{ (\log d + 2 \lambda_1 + 3 \epsilon ) / (\lambda_2 - 4 \epsilon)} . $$
Since $(r_p)_p$ tends to $0$ and  $\hat{\nu}(\hat{\Omega}_{\epsilon}) > 0$, we get 
$$\underline{d_{T,Z}} (\nu )\leq \frac{ \log d + 2 \lambda_1 +  3 \epsilon }{\lambda_2 - 4 \epsilon} =: \frac{ \log d}{\lambda_2} + 2 \frac{\lambda_1 }{\lambda_2} + O_3(\epsilon) .$$
One can prove  
$$\underline{d_{T,W}} (\nu) \leq \frac{  \log d + 2 \lambda_2 + 3 \epsilon }{\lambda_2 - 4 \epsilon} =: \frac{\log d}{\lambda_2} + 2 + O_4(\epsilon)$$
in a similar way.

\subsection{Monge-Amp\`ere mass} \label{section Monge-Amp\`ere}

We prove Proposition~\ref{mesuressubmersion}.
Let $x \in \Sup \mu$, let $V$ be a neighbourhood of $x$ and let $Z : V \to \bbc$ be a holomorphic coordinate (submersion) on $V$. We want to prove that the positive measure $T \wedge \frac{i}{2} d Z \wedge d\bar{Z}$ is not the zero measure on $V$. With no loss of generality, we can assume that $x=(0,0)$, $V = \bbd(2) \times \bbd(2)$ and $Z(z,w) = z$. Let also $T = \ddc G$ on $V$, where $G$ is a continuous psh function. We denote  $\sigma_z(u) := (z,u)$.

\begin{lemme} \label{potentielharmonique}
If $(T \wedge  \frac{i}{2} d Z \wedge d \bar{Z})(\bbd(2) \times \bbd(2)) = 0$, then $G \circ \sigma_z$ is harmonic on $\bbd$ for every $z \in \bbd$.
\end{lemme}

\begin{proof} Let $z_0 \in \bbd$ and let $\varphi \in C_0^{\infty}(\bbd)$ be a test function. Let $\psi \in C_0^{\infty}(\bbd^2)$ such that $\psi \circ \sigma_{z_0} = \varphi$ on $\bbd$. According to Proposition \ref{tranchesdecourant}, we have
$$ \left( T\wedge \frac{i}{2} dZ \wedge d \bar{Z} \right) (\psi) = \int_{z \in \bbd} \left( \int_{w \in \bbd} (G \circ \sigma_z)(w) \times \Delta (\psi \circ \sigma_z)(w) \dd \Leb(w) \right) \dd \Leb(z),$$
which is equal to zero by our assumption. Since the measurable function
$$z \mapsto \int_{w \in \bbd}  (G \circ \sigma_z)(w) \times \Delta (\psi \circ \sigma_z)(w)\dd \Leb(w) $$
is non negative, there exists $A \subset \bbd$ such that $\Leb(A) = \Leb( \bbd)$ and
\begin{equation} \label{eq027}
\forall z \in A, \quad \int_{w \in \bbd} (G \circ \sigma_z)(w) \times \Delta (\psi \circ \sigma_z)(w) \dd \Leb(w) = 0.
\end{equation}
Let us extend this property to every $z \in \bbd$. Since $A$ is a dense subset of $\bbd$, there exists a sequence  $(z_n)_n$ of points in $A$ which converges to $z$. Using  (\ref{eq027}), we get
\begin{equation} \label{eq028}
\forall n \geq 1, \quad \int_{w \in \bbd} (G \circ \sigma_{z_n})(w) \times  \Delta (\psi \circ \sigma_{z_n})(w)  \dd \Leb(w) = 0.
\end{equation}
Since $G$ is continuous on $\bar{\bbd^2}$ and $\psi$ is smooth on $\bar{\bbd^2}$,  $G$ and $\ddc \psi$ are uniformly continuous on $\bar{\bbd^2}$. This implies that $G \circ \sigma_{z_n}$ uniformly converges to $G \circ \sigma_{z}$ on $\bbd$ and that  $\Delta (\psi \circ \sigma_{z_n})$ uniformly converges to $\Delta (\psi \circ \sigma_{z})$ on $\bbd$. Taking the limits in (\ref{eq028}), we get
$$\forall z \in \bbd, \quad \int_{w \in \bbd} (G \circ \sigma_{z})(w) \times \Delta (\psi \circ \sigma_z)(w) \dd \Leb(w) = 0 .$$
In particular, we obtain using $\psi \circ \sigma_{z_0} = \varphi$:
$$\int_{w \in \bbd} (G \circ \sigma_{z_0})(w) \times \Delta  \varphi (w) \dd \Leb(w) = 0 .$$
This holds for every $\varphi \in C_0^{\infty}(\bbd)$, hence the function $G \circ \sigma_{z_0}$ is harmonic on $\bbd$. 
\end{proof}

Now we use the following result, see \cite[Lemme IV.1.1]{bri97} and \cite[Section A.10]{sib99}.

\begin{thm}[Briend] \label{disquesBriend}
Let $G$ be a continuous psh function on $\bbd(2) \times \bbd(2)$. Let $E$ be the set of points $p \in \bbd(\frac{1}{4}) \times \bbd(\frac{1}{4})$ such that there exists a holomorphic disc $\sigma_p : \bbd \to \bbd(2) \times \bbd(2)$ satisfying
\begin{enumerate}
\item the boundary of $\sigma_p$ is outside $\bbd(\frac{1}{2}) \times \bbd(\frac{1}{2})$,
\item $G \circ \sigma_p$ is harmonic $\bbd$. 
\end{enumerate}
Then $(\ddc G \wedge \ddc G)(E) = 0$.
\end{thm}

In our situation, $\bbd(\frac{1}{4}) \times \bbd(\frac{1}{4}) = E$ since one can take for $\sigma_p$ the discs $\sigma_z : \bbd \to \bbd \times \bbd,~u \mapsto (z,u)$. Indeed, the boundary of $\sigma_z$ is contained in $\set{z} \times \partial \bbd$ and $G \circ \sigma_z$ is harmonic on $\bbd$ according to Lemma \ref{potentielharmonique}. Theorem \ref{disquesBriend} then gives:
\begin{equation} \label{eqddc}
(\ddc G \wedge \ddc G) (\bbd(\frac{1}{4}) \times \bbd(\frac{1}{4})) = 0,
\end{equation}
which contradicts $x=0 \in \Sup \mu = \Sup ( \ddc G \wedge \ddc G)$.

\section{Upper bound for the dimension of dilating measures} \label{MDMD}

We prove Theorem \ref{resultat3}. We shall take the proof of Theorem \ref{resultat2} and use Proposition \ref{unifdimcourant}. Let $\epsilon > 0$ and let $\hat \Lambda_\epsilon$ and $\hat \Delta^n_\epsilon$ be the sets  defined in Section \ref{sectunif}. The set $\hat{\Omega}_{\epsilon}^n$ has been defined in Section \ref{dimensionfinieducourant}, it satisfies $\hat{\nu}(\hat{\Omega}_{\epsilon}^n) \geq 1 - 2 \delta$. Hence we have $\hat{\nu}(\hat{\Lambda}_{\epsilon} \cap  \hat \Delta^n_\epsilon \cap \hat{\Omega}_{\epsilon}^n) \geq 1 - 3 \delta$ for every $n \geq N_\epsilon$. Now let $K_n$ be the unique integer satisfying
\begin{equation} \label{eq031}
 \eta_1 e^{-n \lambda_1 - 4 n \epsilon} e^{-\lambda_2 + 3 \epsilon} \leq  \frac{1}{m_0} e^{- K_n \lambda_2 + 3 K_n \epsilon}  \leq \eta_1 e^{-n \lambda_1 - 4 n \epsilon}  .
\end{equation}
so that  $K_n \simeq n \lambda_1 / \lambda_2$. Let $\{ x_1 , \cdots , x_{N_{n,2}} \}$ be a $(n,\eta_1/4)$-separated subset of $A_n := \pi_0(\hat{\Lambda}_{\epsilon} \cap  \hat \Delta^n_\epsilon  \cap  \hat{\Omega}_{\epsilon}^{K_n})$ provided by  Lemma \ref{pointss\'epar\'es2}. We have for every $n \geq \max \{ N_\epsilon , n_{1-3\delta} \} $:
\begin{equation} \label{estimNn2}
 N_{n,2} \geq \nu(A_n) e^{n h_{\nu} - 2 n \epsilon} \geq (1-3\delta) e^{n h_{\nu} - 2 n \epsilon}.
\end{equation}
Then, for every $x_i$, let $\{ y^i_{1}, \dots, y^i_{ M_n}\}$  be a $2 e^{-n \lambda_1 - 4 n \epsilon}$-separated subset of $B_n(x_i,\eta_1 / 4) \cap A_n$ provided by Lemma \ref{pointss\'epar\'es3}. The cardinality of this set satisfies:
\begin{equation} \label{estimMn}
M_n \geq e^{ -n h_{\nu} - 2 n \epsilon} \left( \frac{1}{2 \eta_1} e^{n \lambda_1 + 4 n \epsilon} \right)^{\underline{d_{\nu}} - \epsilon}.
\end{equation}
For every $j \in \set{1, \dots, M_n}$, we set $\hat{y}^i_j \in \hat{\Lambda}_{\epsilon} \cap \hat \Delta^n_\epsilon \cap \hat{\Omega}^{K_n}_{\epsilon}$ such that $y^i_j = \pi_0(\hat{y}^i_j)$. Then we follow the proof of Theorem \ref{resultat2} until the inequality (\ref{tronccommun}):
\begin{equation} \label{tronccommun2}
 d^n \geq \sum_{i=1}^{N_{n,2}} \sum_{j=1}^{M_n} e^{2n\lambda_1 - 4 n \epsilon}  \left[ T \wedge ( \frac{i}{2} dZ_{\hat{y}^i_j}^{\epsilon} \wedge d \bar{Z_{\hat{y}^i_j}^{\epsilon}} ) \right] (B_{y^i_j}(\eta_1 e^{-n \lambda_1 - 4 n \epsilon})) .
\end{equation}
We want to apply Proposition \ref{unifdimcourant}. According to (\ref{eq031}),
$$B_{y^i_j}(\eta_1 e^{-n \lambda_1 - 4 n \epsilon}) \supset B_{y^i_j} \left( \frac{1}{m_0} e^{- K_n \lambda_2 + 3 K_n \epsilon} \right) . $$
We apply the positive measure $T \wedge ( \frac{i}{2} dZ_{\hat{y}^i_j}^{\epsilon} \wedge d \bar{Z_{\hat{y}^i_j}^{\epsilon}} )$ to this inclusion. Since $\hat{y}^i_j \in \hat{\Omega}_{\epsilon}^{K_n}$, we deduce from Proposition \ref{unifdimcourant} that for every $n$ satisfying $n \geq N_\epsilon$ and $K_n \geq N_\epsilon$:
$$ \left[ T \wedge ( \frac{i}{2} dZ_{\hat{y}_{i,j}}^{\epsilon} \wedge d \bar{Z_{\hat{y}_{i,j}}^{\epsilon}} ) \right] (B_{y^i_j}(\eta_1 e^{-n \lambda_1 - 4 n \epsilon})) \geq \frac{1}{d^{K_n}} e^{- 2 K_n \lambda_1 - 2 K_n \epsilon} \frac{1}{q_0} . $$
We infer from (\ref{tronccommun2}) that for every $n$ satisfying $n\geq N_\epsilon$ and $K_n \geq N_\epsilon$,
\begin{equation} \label{eq020}
d^n \geq N_{n,2} \cdot M_n \cdot e^{2n\lambda_1 - 4 n \epsilon} \frac{1}{d^{K_n}} e^{- 2 K_n \lambda_1 - 2 K_n \epsilon} \frac{1}{q_0} .
\end{equation}
Now we use the upper bounds for $N_{n,2}$ and $M_n$ given by (\ref{estimNn2}) and (\ref{estimMn}):
$$ d^{n + K_n} \geq (1-3\delta) e^{n h_{\nu} -  2 n \epsilon} \cdot e^{ -n h_{\nu} - 2 n \epsilon} \left( \frac{1}{2 \eta_1} e^{n \lambda_1 + 4 n \epsilon} \right)^{\underline{d_{\nu}} - \epsilon} \cdot e^{2n\lambda_1 - 4 n \epsilon}  e^{- 2 K_n \lambda_1 - 2 K_n \epsilon} \frac{1}{q_0} . $$
If $C_1(\epsilon) := (1-3\delta) / q_0 (2 \eta_1)^{\underline{d_{\nu}} - \epsilon}$, we get:
$$ \log d + \frac{K_n}{n} \log d \geq {1 \over n } \log C_1(\epsilon)  - 8 \epsilon + (\lambda_1 + 4 \epsilon)(\underline{d_{\nu}} - \epsilon)+ 2 \lambda_1 - 2 \frac{K_n}{n} (\lambda_1 + \epsilon) . $$
By using (\ref{eq031}), we have 
$$  \log d + \frac{\lambda_1 + 4 \epsilon}{\lambda_2 - 3 \epsilon} \log d \geq {  1 \over n} \log C_2(\epsilon) - 8 \epsilon + (\lambda_1 + 4 \epsilon)(\underline{d_{\nu}} - \epsilon)+ 2 \lambda_1 - 2  \frac{\lambda_1 + 4 \epsilon}{\lambda_2 - 3 \epsilon} (\lambda_1 + \epsilon) , $$
where $C_2(\epsilon)$ is another constant. Letting $n$ tend to $+ \infty$ and then $\epsilon$ to $0$, we get 
$$  \underline{d_{\nu}} \leq  \frac{\log d}{\lambda_1}  + \frac{\log d}{\lambda_2}  + 2 \left( \frac{\lambda_1}{\lambda_2} - 1 \right) .$$
To obtain the other upper bound, we use the analogue of (\ref{tronccommun2}) for $W$. Applying Proposition \ref{unifdimcourant} with respect to $W$, we obtain instead of (\ref{eq020}):
$$d^n \geq N_{n,2} \cdot M_n \cdot e^{2n\lambda_2 - 4 n \epsilon} \frac{1}{d^{K_n}} e^{- 2 K_n \lambda_2 - 2 K_n \epsilon} \frac{1}{q_0} .$$
Then we get
$$  \underline{d_{\nu}} \leq \frac{\log d}{\lambda_1}  + \frac{\log d}{\lambda_2} + 2 \left( 1 - \frac{\lambda_2}{\lambda_1} \right) ,$$
which completes the proof of Theorem \ref{resultat3}.

\section{Appendix}\label{appendice}

\subsection{Dimension of measures}

\begin{propo} \label{propdimmesuresabscon}
Let $\nu_1$ and $\nu_2$ be two probability measures on $\bbp^2$ such that $\nu_1 << \nu_2$. Then for $\nu_1$-almost every $x \in \bbp^2$, we have:
$$ \underline{d_{\nu_1}}(x) \geq \underline{d_{\nu_2}}(x) \quad \text{ and } \quad \bar{d_{\nu_1}}(x) \geq \bar{d_{\nu_2}}(x) . $$
\end{propo}

\begin{proof}
Let $\varphi \in L^1(\nu_2)$ such that $\nu_1(A) = \int_{A} \varphi \nu_2$ for every Borel set $A$ of $\bbp^2$. Using the dominated convergence Theorem, 
$$ \lim_{M \to \infty} \int_{\bbp^2} 1_{\lbrace  \varphi \leq M \rbrace}  \varphi \dd \nu_2 = \int_{\bbp^2} \varphi \dd \nu_2 = 1 . $$
For every $n \geq 1$, we let $M_n$ satisfy $\int_{\bbp^2} 1_{\lbrace  \varphi \leq M_n \rbrace}  \varphi d \nu_2 \geq 1 - \frac{1}{n}$. By the Lebesgue density Theorem, for $\nu_1$-almost every $x$ in $\lbrace  \varphi \leq M_n \rbrace$, we have
$$\lim_{r \to 0} \frac{\nu_1(B_x(r) \cap \lbrace  \varphi \leq M_n \rbrace)}{\nu_1(B_x(r))}  = 1 . $$
Then for every $r$ small enough, we have
$$ \frac{1}{2} \nu_1(B_x(r)) \leq \nu_1(B_x(r) \cap \lbrace  \varphi \leq M_n \rbrace) = \int_{B_x(r) \cap \lbrace  \varphi \leq M_n \rbrace} \varphi \, d\nu_2 \leq M_n  \int_{B_x(r)} d\nu_2.$$
And thus $\nu_1(B_x(r)) \leq 2 M_n \nu_2(B_x(r))$. We deduce that
$$\underline{d_{\nu_1}}(x) \geq \underline{d_{\nu_2}}(x) \text{ and } \bar{d_{\nu_1}}(x) \geq \bar{d_{\nu_2}}(x) $$
for $\nu_1$-almost every $x \in \lbrace  \varphi \leq M_n \rbrace$. We end with $\nu_1(\cup_{n \in \bbn} \set{  \varphi \leq M_n }) = 1$. \end{proof}

Now we take the notations of Section \ref{DD}. 

\begin{propo} \label{minmesures} Let $S$ be a $(1,1)$-closed positive current on $\bbp^2$. Let $x \in \bbp^2$ and let $(Z,W)$ be holomorphic coordinates near $x$. Then
$$ \underline{d_S}(x) = \min \set{ \underline{d_{S,Z}}(x),\underline{d_{S,W}}(x) }, \quad \bar{d_S}(x) = \min \set{ \bar{d_{S,Z}}(x),\bar{d_{T,W}}(x) }.$$
\end{propo}

\begin{proof} Let us set $\sigma_{S,Z} = S \wedge (\frac{i}{2} d Z \wedge d \bar{Z})$ and $\sigma_{S,W} = S \wedge (\frac{i}{2} d W \wedge d \bar{W})$.
There exists $c > 0$ such that $\frac{1}{c}(\sigma_{S,Z} + \sigma_{S,W}) \leq \sigma_S \leq c (\sigma_{S,Z} + \sigma_{S,W})$ on a neighbourhood of $x$, see \cite{agbook} Chapter III, \textsection 3. We deduce for every $r$ small enough
$$\frac{1}{c} \max \left[ \sigma_{S,Z}(B_{x}(r)), \sigma_{S,W}(B_{x}(r)) \right] \leq \sigma_S(B_{x}(r)) \leq  2 c \max \left[ \sigma_{S,Z}(B_{x}(r)), \sigma_{S,W}(B_{x}(r)) \right]  . $$
We finish by observing that the local dimension of the maximum of two measures is equal to the minimum of these two dimensions, since one divides by $\log r$ which is negative.
\end{proof}

\subsection{Cohomology and slices}

We refer to Sections 1.2 and A.3 of Dinh-Sibony's book \cite{dinsib10}. 

\begin{propo}\label{cohomologiecourants} Let $S$ be a $(1,1)$-closed positive current of $\bbp^2$ of mass $1$. Let $\omega$ be the Fubini-Study form on $\bbp^2$ and let $f : \bbp^2 \to \bbp^2$ be an endomorphism of degree $d$. Then,
$$ \int_{\bbp^2} (f^n)_* S  \wedge \omega  = \int_{\bbp^2} S  \wedge (f^n)^* \omega  = d^n.$$
\end{propo}
\begin{proof} The first equality comes from the definition of duality. We show the second one.
By using $f^* \omega = d \cdot \omega + \ddc u$, where $u$ is a smooth function on $\bbp^2$, we obtain by induction 
$$ (f^n)^* \omega = d^n \omega + \ddc v_n,$$
where $v_n := (d^{n-1} \cdot u + \dots + d \cdot u \circ f^{n-2} + u \circ f^{n-1})$. Hence
$$ \int_{\bbp^2} S  \wedge (f^n)^* \omega = \int_{\bbp^2} S \wedge \left( d^n \omega + \ddc v_n \right) .   $$
Since $S$ is a closed current of mass $1$, we have $\int_{\bbp^2} S \wedge \ddc v_n  = 0$ and
  $\int_{\bbp^2} S \wedge d^n \omega = d^n$. 
  \end{proof}

\begin{propo} \label{tranchesdecourant} Let $G$ be a continuous psh function on $\bbd^2$ and let $S = \ddc G$. Let $(Z,W)$ be the coordinates on $\bbd^2$ and let $\phi \in C^{\infty}_0(\bbd^2)$. Then
$$ S \wedge \frac{i}{2} dZ \wedge d\bar{Z} (\phi)  = \int_{z \in \bbd} \left( \int_{w \in \bbd} G_{z}(w) \times \Delta \phi_{z}(w) \dd \Leb(w) \right) \dd \Leb(z)= \int_{z \in \bbd} (\sigma_{z}^*S)(\phi_z) d \Leb (z)  , $$
where $\sigma_{z} : u \mapsto (z,u)$, $G_{z} := G \circ \sigma_{z}$ and $\phi_{z} := \phi \circ \sigma_{z}$.
\end{propo}

\begin{proof}
By definition, 
$$ S \wedge \frac{i}{2} dZ \wedge d\bar{Z} (\phi) = \ddc G ( \phi  \frac{i}{2} dZ \wedge d \bar{Z}) = \int_{\bbd^2} G . \ddc ( \phi \times  \frac{i}{2} dZ \wedge d \bar{Z}) .$$
The computation
$$ \ddc ( \phi \times  \frac{i}{2} dZ \wedge d \bar{Z}) =4  (\frac{\partial^2 \phi}{\partial w \partial \bar{w}}) \frac{i}{2} dW \wedge d\bar{W} \wedge  \frac{i}{2} dZ \wedge d \bar{Z} = 4 (\frac{\partial^2 \phi}{\partial w \partial \bar{w}}) \dd \Leb (z,w)$$
allows to write
\begin{align*}
S \wedge \frac{i}{2} dZ \wedge d\bar{Z} (\phi) & = \int_{(z,w) \in \bbd^2} G(z,w) \times 4 \frac{\partial^2 \phi}{\partial w \partial \bar{w}}(z,w) \dd \Leb(z,w) \\
  & = \int_{z \in \bbd}\left(  \int_{w  \in \bbd} G_z(w) \times  \Delta \phi_z (w)  \dd \Leb(w) \right) \dd \Leb(z) . 
\end{align*}
Finally, the quantity in brackets is equal to $(\Delta G_{z}) (\phi_{z}) = (\sigma_{z}^*S)(\phi_{z})$. 
\end{proof}

\bibliographystyle{plain}
\bibliography{biblio}

$ $ \\

\noindent {\footnotesize Christophe Dupont and Axel Rogue}\\
{\footnotesize Universit\'e de Rennes 1}\\
{\footnotesize IRMAR, CNRS UMR 6625}\\
{\footnotesize Campus de Beaulieu, B\^at. 22-23}\\
{\footnotesize F-35042 Rennes Cedex, France}\\
{\footnotesize christophe.dupont@univ-rennes1.fr}\\
{\footnotesize axel.rogue@univ-rennes1.fr}\\

\end{document}